\documentclass[11pt]{amsart}
 \usepackage{amsmath,amssymb,amsfonts,amsthm,amscd,textcomp,times,booktabs}
   \usepackage[dvips]{graphicx}
     \usepackage{braket}
  \usepackage{color,float}
  \usepackage{bm}
  \usepackage[all]{xy}
\newtheorem{theorem}{Theorem}

\newtheorem{proposition}[theorem]{Proposition}
\newtheorem{lemma}[theorem]{Lemma}
\newtheorem{definition}[theorem]{Definition}
\newtheorem{corollary}[theorem]{Corollary}
\newtheorem{remark}[theorem]{Remark}

\numberwithin{equation}{section}
\numberwithin{theorem}{section}

\newcommand{\Z}{\ensuremath{\mathbb{Z}}}
\newcommand{\R}{\ensuremath{\mathbb{R}}}
\newcommand{\C}{\ensuremath{\mathbb{C}}}
\newcommand{\I}{\ensuremath{\sqrt{-1}}}

\newcommand{\GL}{\ensuremath{\mathrm{GL}}}

\newcommand{\SO}{\ensuremath{\mathrm{SO}}}
\newcommand{\Spin}{\ensuremath{\mathrm{Spin(7)}}}
\newcommand{\SU}{\ensuremath{\mathrm{SU}}}

\newcommand{\del}{\ensuremath{\mathrm{\partial}}}
\newcommand{\delbar}{\ensuremath{\mathrm{\overline{\partial}}}}
\newcommand{\der}{\ensuremath{\mathrm{d}}}
\newcommand{\id}{\ensuremath{\mathrm{id}}}
\newcommand{\din}{\rotatebox{90}{\ensuremath{\in}}}

\newcommand{\veq}{\rotatebox{90}{\ensuremath{=}}}

\newcommand{\vol}{\ensuremath{\mathrm{vol}}}

\newcommand{\norm}[1]{\ensuremath{\left| #1 \right|}}
\newcommand{\Norm}[1]{\ensuremath{\left\| #1 \right\|}}
\newcommand{\restrict}[2]{\ensuremath{\left. #1 \right|_{#2}}}

\DeclareMathOperator{\Image}{Im}
\DeclareMathOperator{\Real}{Re}
\DeclareMathOperator{\Ker}{Ker}

\DeclareMathOperator{\rank}{rank}

\DeclareMathOperator*{\Hol}{Hol}

\def\P{\mathbb{P}}

\begin{document}
\title{Doubling construction of Calabi-Yau threefolds}

\author{Mamoru Doi}

\address{}
 \email{doi.mamoru@gmail.com}

\author{Naoto Yotsutani}

\address{School of Mathematical Sciences at Fudan University,
Shanghai, 200433, P. R. China}
\email{naoto-yotsutani@fudan.edu.cn}

 \dedicatory{}
\date{\today}
\maketitle

\noindent{\bfseries Abstract.}
We give a differential-geometric construction and examples of Calabi-Yau threefolds, at least one of which is {\it{new}}.
Ingredients in our construction are {\it admissible pairs}, which were dealt with by Kovalev in \cite{K03}
and further studied by Kovalev and Lee in \cite{KL11}.
An admissible pair $(\overline{X},D)$ consists of
a three-dimensional compact K\"{a}hler manifold $\overline{X}$ and
a smooth anticanonical $K3$ divisor $D$ on $\overline{X}$.
If two admissible pairs $(\overline{X}_1,D_1)$ and $(\overline{X}_2,D_2)$ satisfy
the {\it gluing condition}, we can glue $\overline{X}_1\setminus D_1$ and
$\overline{X}_2\setminus D_2$ together to obtain a Calabi-Yau threefold $M$.
In particular, if $(\overline{X}_1,D_1)$ and $(\overline{X}_2,D_2)$
are identical to an admissible pair $(\overline{X},D)$,
then the gluing condition holds automatically, so that we can {\it always} construct
a Calabi-Yau threefold from a {\it single} admissible pair $(\overline{X},D)$
by {\it doubling} it.
Furthermore, we can compute all Betti and Hodge numbers of the resulting Calabi-Yau threefolds
in the doubling construction.
\section{Introduction}
The purpose of this paper is to give a gluing construction and examples
of Calabi-Yau threefolds.
Before going into details, we recall some historical background behind
our gluing construction.

The gluing technique is used in constructing many compact manifolds
with a special geometric structure.
In particular, it is effectively used in constructing compact manifolds
with exceptional holonomy groups $G_2$ and $\Spin$,
which are also called compact $G_2$- and $\Spin$- manifolds respectively.
The first examples of compact $G_2$- and $\Spin$- manifolds were obtained by Joyce
using Kummer-type constructions in a series of his papers \cite{J96S,J96G,J99}.
Also, Joyce gave a second construction of compact $\Spin$-manifolds using compact
four-dimensional K\"{a}hler orbifolds with an antiholomorphic involution.
These constructions are based on the resolution of certain singularities
by replacing neighborhoods of singularities with ALE-type manifolds. 
Later, Clancy studied in \cite{C11} such compact K\"{a}hler orbifolds systematically and
constructed more new examples of compact $\Spin$-manifolds using Joyce's second construction.

On the other hand, Kovalev gave another construction of compact $G_2$-manifolds in \cite{K03}.
Beginning with a Fano threefold $\overline{W}$ with a smooth anticanonical $K3$ divisor $D$,
he showed that if we blow up $\overline{W}$ along a curve representing $D\cdot D$
to obtain $\overline{X}$, then $\overline{X}$ has an anticanonical divisor isomorphic to $D$
(denoted by $D$ again) with the holomorphic normal bundle $N_{D/\overline{X}}$ trivial.
Then $\overline{X}\setminus D$ admits an asymptotically cylindrical Ricci-flat K\"{a}hler metric.
(We call such $(\overline{X},D)$ an {\it admissible pair of Fano type}.)
Also, Kovalev proved that
if two admissible pairs $(\overline{X}_1,D_1)$ and $(\overline{X}_2,D_2)$
satisfy a certain condition called the {\it matching condition},
we can glue together $(\overline{X}_1\setminus D_1)\times S^1$ and $(\overline{X}_2\setminus D_2)\times S^1$
along their cylindrical ends in a {\it twisted} manner to obtain a compact $G_2$-manifold.
In this construction, Kovalev found many new examples of $G_2$-manifolds
using the classification of Fano threefolds by Mori and Mukai \cite{MM81, MM03}.
Later, Kovalev and Lee \cite{KL11} found admissible pairs of another type
(which are said to be {\it admissible pairs of non-symplectic type}) and constructed
new examples of compact $G_2$-manifolds. They used the classification of $K3$ surfaces
with a non-symplectic involution by Nikulin \cite{N80}.

In our construction, we begin with two admissible pairs $(\overline{X}_1,D_1)$
and $(\overline{X}_2,D_2)$ as above.
Then each $(\overline{X}_i\setminus D_i)\times S^1$ has a natural asymptotically cylindrical
torsion-free $G_2$-structure $\varphi_{i,{\rm cyl}}$ using the existence result of
an asymptotically cylindrical Ricci-flat K\"{a}hler form on $\overline{X}_i\setminus D_i$.
Now suppose $\overline{X}_1\setminus D_1$ and $\overline{X}_2\setminus D_2$
have the same {\it asymptotic model}, which is ensured by the {\it gluing condition}
defined later.
Then as in Kovalev's construction, we can glue together
$(\overline{X}_1\setminus D_1)\times S^1$
and $(\overline{X}_2\setminus D_2)\times S^1$,
but in a {\it non}-twisted manner to obtain $M_T\times S^1$. In short,
we glue together $\overline{X}_1\setminus D_1$ and $\overline{X}_2\setminus D_2$
along their cylindrical ends $D_1\times S^1\times (T-1,T+1)$ and
$D_2\times S^1\times (T-1,T+1)$,
and then take the product with $S^1$.
Moreover, we can glue together torsion-free $G_2$-structures to construct a
$\der$-closed $G_2$-structure $\varphi_T$ on $M_T\times S^1$.
Using the analysis on torsion-free $G_2$-structures,
we shall prove that $\varphi_T$ can be deformed into a torsion-free $G_2$-structure
for sufficiently large $T$, so that
the resulting compact manifold $M_T\times S^1$
admits a Riemannian metric with holonomy contained in $G_2$.
But if $M=M_T$ is simply-connected, then $M$ must have holonomy $\SU (3)$ according to the
Berger-Simons classification of holonomy groups of Ricci-flat Riemannian manifolds.
Hence this $M$ is a Calabi-Yau threefold.

For two given admissible pairs $(\overline{X}_1,D_1)$ and $(\overline{X}_2,D_2)$,
it is difficult to check in general whether the gluing condition holds or not.
However, if $(\overline{X}_1,D_1)$ and $(\overline{X}_2,D_2)$
are identical to an admissible pair $(\overline{X},D)$,
then the gluing condition holds automatically.
Therefore we can {\it always} construct a Calabi-Yau threefold from a {\it single} admissible pair
$(\overline{X},D)$ by {\it doubling} it.

Our doubling construction has another advantage in computing Betti and Hodge numbers of the
resulting Calabi-Yau threefolds $M$.
To compute Betti numbers of $M$, it is necessary to find out the intersection of
the images of the homomorphisms
$H^2(X_i,\R )\longrightarrow H^2(D_i,\R )$ for $i=1,2$ induced by the inclusion
$D_i\times S^1\longrightarrow X_i$, where we denote $X_i=\overline{X}_i\setminus D_i$.
In the doubling construction, the above two homomorphisms are identical, and
the intersection of their images is the same as each one.

With this construction, we shall give $123$ topologically distinct Calabi-Yau threefolds
($59$ examples from admissible pairs of Fano type and $64$ from those of non-symplectic type).
Moreover, $54$ of the Calabi-Yau threefolds from admissible pairs of non-symplectic type form mirror pairs ($24$ mirror pairs
and $6$ self mirrors). In a word, we construct Calabi-Yau threefolds and their mirrors from $K3$ surfaces. This construction was previously investigated by Borcea and Voisin \cite{BV97} using algebro-geometric methods. Thus, our doubling construction from non-symplectic type
can be interpreted as a differential-geometric analogue of the Borcea-Voisin construction. 
Furthermore, the remaining $10$ examples from non-symplectic type
 contain at least one new example. See `Discussion' in Section \ref{subset:non-symp} for more details. Meanwhile,
$59$ examples from admissible pairs of Fano type are essentially the same Calabi-Yau threefolds constructed by Kawamata and Namikawa \cite{KM94} and later developed by Lee \cite{L10} using normal crossing varieties. Hence our construction from Fano type
provides a differential-geometric interpretation of Lee's construction \cite{L10}.

This paper is organized as follows.
Section 2 is a brief review of $G_2$-structures. In Section 3 we establish our
gluing construction of Calabi-Yau threefolds from admissible pairs.
The rest of the paper is devoted to constructing examples
and computing Betti and Hodge numbers of Calabi-Yau threefolds
obtained in our doubling construction.
The reader who is not familiar with analysis can
check Definition \ref{def:admissible} of admissible pairs, go to Section 3.4 where
the gluing theorems are stated, and then proceed to Section 4,
skipping Section 2 and the rest of Section 3.
In Section 4 we will find a formula for computing Betti numbers of the resulting
Calabi-Yau threefolds $M$ in our doubling construction.
In Section 5, we recall two types of admissible pairs and rewrite
the formula given in Section 4 to obtain formulas of Betti and Hodge numbers of $M$
in terms of certain invariants which characterize admissible pairs.
Then the last section lists all data of the Calabi-Yau threefolds obtained in our construction.

The first author is mainly responsible for Sections 1--3,
and the second author mainly for Sections 4--6.\\

\noindent{\bfseries Acknowledgements.} This joint work was partially
developed when the second author attended ``2012 Complex Geometry
and Symplectic Geometry Conference" which was held at Ningbo
University, Zhejian Province, China in July, 2012. He thanks
Professor Xiuxiong Chen for an invitation to give a talk in the conference.
Also, he thanks Professor Yuguang Zhang for a helpful comment during the
conference. The authors thank Dr. Nam-Hoon Lee for pointing out that our construction partially covers
the Borcea-Voisin method \cite{BV97}. Also, they would like to thank the referee and Dr. Shintaro Kuroki for valuable comments.
Especially, the referee pointed out a gap of the proof of Lemma \ref{lem:RFonM} and advised them how to correct it, giving a good reference \cite{MA13}.
The second author is partially supported by the China
Postdoctoral Science Foundation Grant, No. 2011M501045 and the Chinese Academy
of Sciences Fellowships for Young International Scientists 2011Y1JB05.

\section{Geometry of $G_2$-structures}
Here we shall recall some basic facts about $G_2$-structures on oriented $7$-manifolds.
For more details, see Joyce's book \cite{J00} .

We begin with the definition of $G_2$-structures on oriented vector spaces of dimension $7$.
\begin{definition}\rm
Let $V$ be an oriented real vector space of dimension $7$.
Let $\{\bm{\theta}^1,\dots ,\bm{\theta}^7\}$ be an oriented basis of $V$.
Set
\begin{equation}\label{eq:phi0-g0}
\begin{aligned}
\bm{\varphi}_0&=\bm{\theta}^{123}+\bm{\theta}^{145}+\bm{\theta}^{167}+\bm{\theta}^{246}
-\bm{\theta}^{257}-\bm{\theta}^{347}-\bm{\theta}^{356},\\
\bm{g}_0&=\sum_{i=1}^7\bm{\theta}^i\otimes\bm{\theta}^i,
\end{aligned}
\end{equation}
where $\bm{\theta}^{ij\dots k}=\bm{\theta}^i\wedge\bm{\theta}^j\wedge\dots\wedge\bm{\theta}^k$.
Define the $\GL_+(V)$-orbit spaces
\begin{align*}
\mathcal{P}^3(V)&=\Set{a^*\bm{\varphi}_0|a\in\GL_+(V)},\\
\mathcal{M}et(V)&=\Set{a^*\bm{g}_0|a\in\GL_+(V)}.
\end{align*}
We call $\mathcal{P}^3(V)$ the set of {\it positive $3$-forms}
(also called the set of {\it $G_2$-structures} or {\it associative $3$-forms}) on $V$.
On the other hand, $\mathcal{M}et(V)$ is the set of positive-definite inner products on $V$,
which is also a homogeneous space isomorphic to $\GL_+(V)/\SO(V)$, where $\SO(V)$ is defined by
\begin{equation*}
\SO(V)=\Set{a\in\GL_+(V)|a^*\bm{g}_0=\bm{g}_0}.
\end{equation*}

Now the group $G_2$ is defined as the isotropy of the action of $\GL(V)$ (in place of $\GL_+(V)$)
on $\mathcal{P}^3(V)$ at $\bm{\varphi}_0$:
\begin{equation*}
G_2=\Set{a\in\GL(V)|a^*\bm{\varphi}_0=\bm{\varphi}_0}.
\end{equation*}
Then one can show that $G_2$ is a compact Lie group of dimension $14$ which is a
Lie subgroup of $\SO(V)$ \cite{Hy90}.
Thus we have a natural projection
\begin{equation}\label{eq:G2metric}
\xymatrix{\mathcal{P}^3(V)\cong\GL_+(V)/G_2\ar@{>>}[r]&\GL_+(V)/\SO(V)\cong\mathcal{M}et(V)},
\end{equation}
so that each positive $3$-form (or $G_2$-structure) $\bm{\varphi}\in\mathcal{P}^3(V)$
defines a positive-definite inner product $\bm{g}_{\bm{\varphi}}\in\mathcal{M}et(V)$ on $V$.
In particular, \eqref{eq:G2metric} maps $\bm{\varphi}_0$ to $\bm{g}_0$ in \eqref{eq:phi0-g0}.
Note that both $\mathcal{P}^3(V)$ and $\mathcal{M}et(V)$ depend only on the orientation of $V$
and are independent of the choice of an oriented basis $\{\bm{\theta}^1,\dots ,\bm{\theta}^7\}$,
and so is the map \eqref{eq:G2metric}. Note also that
\begin{equation*}
\dim_\R\mathcal{P}^3(V)=\dim_\R\GL_+(V)-\dim_\R G_2=7^2-14=35,
\end{equation*}
which is the same as $\dim_\R\wedge^3 V$.
This implies that $\mathcal{P}^3(V)$ is an {\it open} subset of $\wedge^3 V$.
The following lemma is immediate.
\begin{lemma}\label{lem:ep-nbd}
There exists a constant $\rho_* >0$ such that for any
$\bm{\varphi}\in\mathcal{P}^3(V)$, if $\widetilde{\bm{\varphi}}\in\wedge^3
V$ satisfies
$\norm{\widetilde{\bm{\varphi}}-\bm{\varphi}}_{\bm{g}_{\bm{\varphi}}}<\rho_*$,
then $\widetilde{\bm{\varphi}}\in\mathcal{P}^3(V)$.
\end{lemma}

\end{definition}
\begin{remark}\rm
Here is an alternative definition of $G_2$-structures.
But the reader can skip the following.
Let $V$ be an oriented real vector space of dimension $7$ with orientation $\bm{\mu}_0$.
Let $\bm{\Omega}\in\wedge^7 V^*$ be a volume form which is positive
with respect to the orientation $\bm{\mu}_0$.
Then $\bm{\varphi}\in\wedge^3 V^*$ is a positive $3$-form
on $V$ if an inner product $\bm{g}_{\bm{\Omega}, \bm{\varphi}}$ given by
\begin{equation*}
\iota_{\bm{u}}\bm{\varphi}\wedge\iota_{\bm{v}}\bm{\varphi}\wedge\bm{\varphi} 
=6\; \bm{g}_{\bm{\Omega},\bm{\varphi}}(\bm{u},\bm{v})\bm{\Omega}
\quad\text{for }\bm{u},\bm{v}\in V
\end{equation*}
is positive-definite, where $\iota_{\bm{u}}$ denotes interior product
by a vector $\bm{u}\in V$, from which comes the name `positive form'.
Whether $\bm{\varphi}$ is a positive $3$-form
depends only on the orientation $\bm{\mu}_0$ of $V$, and is independent of
the choice of a positive volume form $\bm{\Omega}$.
One can show that if $\bm{\varphi}$ is a positive $3$-form on $(V,\bm{\mu}_0 )$,
then there exists a unique positive-definite inner product $\bm{g}_{\bm{\varphi}}$ such that
\begin{equation*}
\iota_{\bm{u}}\bm{\varphi}\wedge\iota_{\bm{v}}\bm{\varphi}\wedge\bm{\varphi} 
=6\; \bm{g}_{\bm{\varphi}}(\bm{u},\bm{v})\vol_{\bm{g}_{\bm{\varphi}}}
\quad\text{for }\bm{u},\bm{v}\in V,
\end{equation*}
where $\vol_{\bm{\varphi}}$ is a volume form determined by $\bm{g}_{\bm{\varphi}}$ and 
$\bm{\mu}_0$.
The map $\bm{\varphi}\longmapsto \bm{g}_{\bm{\varphi}}$ gives \eqref{eq:G2metric} explicitly.
One can also prove that there exists an othogornal basis $\{\bm{\theta}^1,\dots ,\bm{\theta}^7\}$
with respect to $\bm{g}_{\bm{\varphi}}$ such that $\bm{\varphi}$ and $\bm{g}_{\bm{\varphi}}$ 
are written in the same form as $\bm{\varphi}_0$ and $\bm{g}_0$ in \eqref{eq:phi0-g0}.
\end{remark}
Now we define $G_2$-structures on oriented $7$-manifolds.
\begin{definition}\rm
Let $M$ be an oriented $7$-manifold.
We define $\mathcal{P}^3(M)\longrightarrow M$ to be the fiber bundle whose fiber over $x$ is
$\mathcal{P}^3(T^*_x M)\subset\wedge^3 T^*_x M$. Then
$\varphi\in C^\infty (\wedge^3 T^* M)$ is a {\it positive $3$-form}
(also an {\it associative $3$-form} or a {\it $G_2$-structure}) on $M$ if
$\varphi\in C^\infty (\mathcal{P}^3(M))$, i.e.,
$\varphi$ is a smooth section of $\mathcal{P}^3(M)$.
If $\varphi$ is a $G_2$-structure on $M$, then $\varphi$ induces a Riemannian metric $g_\varphi$
since each $\restrict{\varphi}{x}$ for $x\in M$ induces a positive-definite inner product
$g_{\restrict{\varphi}{x}}$ on $T_x M$. A $G_2$-structure $\varphi$ on $M$ is said to be
{\it torsion-free} if it is parallel with respect to the induced Riemannian metric $g_\varphi$,
i.e., $\nabla_{g_\varphi}\varphi =0$,
where $\nabla_{g_\varphi}$ is the Levi-Civita connection of $g_\varphi$.
\end{definition}
\begin{lemma}\label{lem:ep-nbd_2}
Let $\rho_*$ be the constant given in Lemma $\ref{lem:ep-nbd}$. For
any $\varphi\in\mathcal{P}^3(M)$, if $\widetilde{\varphi}\in
C^\infty (\wedge^3T^*M)$ satisfies $\Norm{\widetilde{\varphi}
-\varphi}_{C^0}<\rho_*$, then
$\widetilde{\varphi}\in\mathcal{P}^3(M)$, where $\Norm{\cdot}_{C^0}$
is measured using the metric $g_{\varphi}$ on $M$.
\end{lemma}
The following result is one of the most important results in the geometry of
the exceptional holonomy group $G_2$,
relating the holonomy contained in $G_2$ with the $\der$- and $\der^*$-closedness
of the $G_2$-structure.
\begin{theorem}[Salamon \cite{S89}, Lemma 11.5]
Let $M$ be an oriented $7$-manifold. Let $\varphi$ be a $G_2$-structure on $M$ and $g_\varphi$
the induced Riemannian metric on $M$.
Then the following conditions are equivalent.
\renewcommand{\labelenumi}{\theenumi}
\renewcommand{\theenumi}{(\arabic{enumi})}
\begin{enumerate}
\item $\varphi$ is a torsion-free $G_2$-structure, i.e., $\nabla_{g_\varphi}\varphi =0$.
\item $\der\varphi =\der *_{g_\varphi}\varphi =0$, where $*_{g_\varphi}$ is the Hodge star
operator induced by $g_\varphi$.
\item $\der\varphi =\der^*_{g_\varphi}\varphi =0$, where
$d^*_{g_\varphi}=-*_{g_\varphi}\der*_{g_\varphi}$ is the formal adjoint operator of $\der$.
\item The holonomy group $\Hol (g_\varphi )$ of $g_\varphi$ is contained in $G_2$.
\end{enumerate}
\end{theorem}

\section{The Gluing Procedure}

\subsection{Compact complex manifolds with an anticanonical divisor}\label{section:CMWAD}
We suppose that $\overline{X}$ is a compact complex manifold of dimension $m$,
and $D$ is a smooth irreducible anticanonical divisor on $\overline{X}$.
We recall some results in \cite{D09}, Sections $3.1$ and $3.2$.
\begin{lemma}\label{lem:coords_on_X}
Let $\overline{X}$ be a compact complex manifold of dimension $m$ and $D$ a smooth irreducible
anticanonical divisor on $\overline{X}$.
Then there exists a local coordinate system
$\{ U_\alpha ,(z_\alpha^1,\dots ,z_\alpha^{m-1},w_\alpha)\}$ on $\overline{X}$
such that
\renewcommand{\labelenumi}{\theenumi}
\renewcommand{\theenumi}{(\roman{enumi})}
\begin{enumerate}
\item $w_\alpha$ is a local defining function of $D$ on $U_\alpha$,
i.e., $D\cap U_\alpha =\{w_\alpha =0\}$, and
\item the $m$-forms $\displaystyle\Omega_\alpha =
\frac{\der w_\alpha}{w_\alpha}\wedge\der z_\alpha^1\wedge\dots\wedge\der z_\alpha^{m-1}$
on $U_\alpha$ together yield a holomorphic volume form $\Omega$ on $X=\overline{X}\setminus D$.
\end{enumerate}
\end{lemma}
Next we shall see that $X=\overline{X}\setminus D$ is a cylindrical manifold whose structure is
induced from the holomorphic normal bundle $N=N_{D/\overline{X}}$ to $D$ in $\overline{X}$,
where the definition of cylindrical manifolds is given as follows.
\begin{definition}\label{def:cyl.mfd}\rm
Let $X$ be a noncompact differentiable manifold of dimension $n$.
Then $X$ is called a {\it cylindrical manifold} or a {\it manifold
with a cylindrical end} if there exists a diffeomorphism $\pi
:X\setminus X_0\longrightarrow\Sigma\times\R_+=\Set{(p,t)|p\in\Sigma
,0<t<\infty}$ for some compact submanifold $X_0$ of dimension $n$
with boundary $\Sigma=\del X_0$. Also, extending $t$ smoothly to $X$ so that
$t\leqslant 0$ on $X\setminus X_0$, we call $t$ a {\it cylindrical
parameter} on $X$.
\end{definition}
Let $(x_\alpha ,y_\alpha)$ be local coordinates on $V_\alpha =U_\alpha\cap D$,
such that $x_\alpha$ is the restriction of $z_\alpha$ to $V_\alpha$ and
$y_\alpha$ is a coordinate in the fiber direction.
Then one can see easily that $\der x_\alpha^1\wedge\dots\wedge\der x_\alpha^{m-1}$ on
$V_\alpha$ together yield a holomorphic volume form $\Omega_D$, which is also called the
{\it Poincar\'{e} residue} of $\Omega$ along $D$.
Let $\Norm{\cdot}$ be the norm of a Hermitian bundle metric on $N$.
We can define a cylindrical parameter $t$ on $N$
by $t=-\frac{1}{2}\log\Norm{s}^2$ for $s\in N\setminus D$.
Then the local coordinates $(z_\alpha ,w_\alpha )$ on $X$ are asymptotic to
the local coordinates $(x_\alpha ,y_\alpha )$ on $N\setminus D$ in the following sense.

\begin{lemma}\label{lem:tub.nbd.thm}
There exists a diffeomorphism $\Phi$ from a neighborhood $V$ of the
zero section of $N$ containing $t^{-1}(\R_+ )$ to a tubular neighborhood of $U$ of $D$
in $X$ such that $\Phi$ can be locally written as
\begin{equation*}
\begin{aligned}
z_\alpha &=x_\alpha + O(\norm{y_\alpha}^2)=x_\alpha +O(e^{-t}),\\
w_\alpha &=y_\alpha + O(\norm{y_\alpha}^2)=y_\alpha +O(e^{-t}),
\end{aligned}
\end{equation*}
where we multiply all $z_\alpha$ and $w_\alpha$ by a single constant
to ensure $t^{-1}(\R_+ )\subset V$ if necessary.
\end{lemma}
Hence $X$ is a cylindrical manifold with the cylindrical parameter $t$
via the diffeomorphism $\Phi$ given in the above lemma.
In particular, when $H^0(\overline{X},\mathcal{O}_{\overline{X}})=0$ and
$N_{D/\overline{X}}$ is trivial, we have a useful coordinate system near $D$.
\begin{lemma}\label{lem:existence_w}
Let $(\overline{X},D)$ be as in Lemma $\ref{lem:coords_on_X}$. If
$H^1(\overline{X},\mathcal{O}_{\overline{X}})=0$ and the normal
bundle $N_{D/\overline{X}}$ is holomorphically trivial, then there
exists an open neighborhood $U_D$ of $D$ and a holomorphic function
$w$ on $U_D$ such that $w$ is a local defining function of $D$ on
$U_D$. Also, we may define the cylindrical parameter $t$ with
$t^{-1}(\R_+)\subset U_D$ by writing the fiber coordinate $y$ of
$N_{D/\overline{X}}$ as $y=\exp (-t-\I\theta )$.
\end{lemma}
\begin{proof}
We deduce from the short exact sequence
\begin{equation*}
\xymatrix@R=-1ex{0\ar[r]&\mathcal{O}_{\overline{X}}\ar[r]&[D]\ar[r]&\restrict{[D]}{D}\ar[r]&0\\
&&
&\veq\mathstrut\;\;\; &\\
&& &N_{D/\overline{X}}\cong\mathcal{O}_D& }
\end{equation*}
the long exact sequence
\begin{equation*}
\xymatrix@R=-1ex{
\cdots\ar[r]&H^0({\overline{X}},[D])\ar@{>>}[r]&H^0(D,N_{D/\overline{X}})\ar[r]
&H^1(\overline{X},\mathcal{O}_{\overline{X}})\ar[r]&\cdots .\\
&&\veq&\veq &\\
&&H^0(D,\mathcal{O}_D)\cong\C &0&
}
\end{equation*}
Thus there exists a holomorphic section $s\in H^0({\overline{X}},[D])$ such that
$\restrict{s}{D}\equiv 1\in H^0(D,N_{D/\overline{X}})$. Setting
$\displaystyle U_D=\Set{x\in {\overline{X}}|s(x)\neq 0}$, we have
$\restrict{[D]}{U_D}\cong\mathcal{O}_{U_D}$, so that there exists a
local defining function $w$ of $D$ on $U_D$.
\end{proof}

\subsection{Admissible pairs and asymptotically cylindrical Ricci-flat K\"{a}hler manifolds}

\begin{definition}\rm
Let $X$ be a cylindrical manifold such that
$\pi :X\setminus X_0\longrightarrow\Sigma\times\R_+=\{(p,t)\}$
is a corresponding diffeomorphism.
If $g_\Sigma$ is a Riemannian metric on $\Sigma$, then it defines a cylindrical metric
$g_{\rm cyl}=g_\Sigma +\der t^2$ on $\Sigma\times\R_+$.
Then a complete Riemannian metric $g$ on $X$ is said to be {\it asymptotically cylindrical}
({\it to }$(\Sigma\times\R_+,g_{\rm cyl})$) if $g$ satisfies
\begin{equation*}
\norm{\nabla_{g_{\rm cyl}}^j(g-g_{\rm cyl})}_{g_{\rm cyl}}\longrightarrow 0
\quad\text{as }t\longrightarrow\infty\quad\text{for all }j\geqslant 0
\end{equation*}
for some cylindrical metric $g_{\rm cyl}=g_\Sigma +\der t^2$,
where we regarded $g_{\rm cyl}$ as a Riemannian metric on $X\setminus X_0$
via the diffeomorphism $\pi$.
Also, we call $(X,g)$ an {\it asymptotically cylindrical manifold} and
$(\Sigma\times\R_+,g_{\rm cyl})$ the {\it asymptotic model} of $(X,g)$.
\end{definition}
\begin{definition}\label{def:admissible}\rm
Let $\overline{X}$ be a compact K\"{a}hler manifold and $D$ a divisor on $\overline{X}$.
Then $(\overline{X},D)$ is said to be an {\it admissible pair} if the following conditions hold:
\begin{enumerate}
\item[(a)] $\overline{X}$ is a compact K\"ahler manifold,
\item[(b)] $D$ is a smooth anticanonical divisor on $\overline{X}$,
\item[(c)] the normal bundle $N_{\overline{X}/ D}$ is trivial, and
\item[(d)] $\overline{X}$ and $X=\overline{X}\setminus D$ are simply-connected.
\end{enumerate}
\end{definition}
From the above conditions, we see that Lemmas \ref{lem:coords_on_X} and
\ref{lem:existence_w} apply to admissible pairs.
Also, from conditions (a) and (b), we see that $D$ is a compact K\"{a}hler manifold
with trivial canonical bundle.
In particular, if $\dim_\C \overline{X}=3$, which case is our main concern, 
then $D$ must be a $K3$ surface (and so cannot be a complex torus). Let us shortly see this.
The short exact sequence $0\longrightarrow K_{\overline{X}}\longrightarrow \mathcal{O}_{\overline{X}}\longrightarrow 
\mathcal{O}_D\longrightarrow 0$
induces the long exact sequence 
\begin{equation*}
\xymatrix@R=-1ex{\cdots\ar[r]&H^1(\overline{X},\mathcal{O}_{\overline{X}})\ar[r]
&H^1(D,\mathcal{O}_D)\ar[r]&H^2(\overline{X},K_{\overline{X}})\ar[r]&\cdots}.
\end{equation*}
Here $H^2(\overline{X},K_{\overline{X}})$ is dual to $H^1(\overline{X},\mathcal{O}_{\overline{X}})$ by the Serre duality
and $H^1(\overline{X},\mathcal{O}_{\overline{X}})\cong H^{0,1}_{\delbar}(\overline{X})$ vanishes from $b^1(\overline{X})=0$.
Thus $H^1(D,\mathcal{O}_D)\cong H^{0,1}_{\delbar}(D)$ also vanishes, so that we have $b^1(D)=0$.
\begin{theorem}[Tian-Yau \cite{TY90}, Kovalev \cite{K03}, Hein \cite{Hn10}]\label{thm:TYKH}
Let $(\overline{X},\omega' )$ be a compact K\"{a}hler manifold
and $m=\dim_\C\overline{X}$. 
If $(\overline{X},D)$ is an admissible pair, then the following is true.

It follows from Lemmas $\ref{lem:coords_on_X}$ and
$\ref{lem:existence_w}$, there exist a local coordinate system
$(U_{D,\alpha} ,(z_\alpha^1,\dots ,z_\alpha^{m-1},w))$ on a
neighborhood $U_D=\cup_\alpha U_{D,\alpha}$ of $D$ and a holomorphic
volume form $\Omega$ on $\overline{X}$ such that
\begin{equation}\label{eq:TYKH_Omega}
\Omega =\frac{\der w}{w}\wedge\der z_\alpha^1\wedge\dots\wedge
\der z_\alpha^{m-1}\quad\text{on }U_{D,\alpha}.
\end{equation}
Let $\kappa_D$ be the unique Ricci-flat K\"{a}hler form on $D$ in the K\"{a}hler class
$[\restrict{\omega'}{D}]$.
Also let $(x_\alpha ,y)$ be local coordinates of $N_{D/\overline{X}}\setminus D$
as in Section $\ref{section:CMWAD}$ and write $y$ as $y=\exp (-t-\I\theta )$.
Now define a holomorphic volume form $\Omega_{\rm cyl}$
and a cylindrical Ricci-flat K\"{a}hler form $\omega_{\rm cyl}$ by
\begin{equation}\label{eq:TYKH_CYcyl}
\begin{aligned}
\Omega_{\rm cyl}&=\frac{\der y}{y}\wedge\der x_\alpha^1\wedge\dots\wedge\der x_\alpha^{m-1}
=(\der t +\I\der\theta)\wedge\Omega_D,\\
\omega_{\rm cyl}&=\kappa_D+\frac{\der y\wedge\der\overline{y}}{\norm{y}^2}
=\kappa_D+\der t\wedge\der\theta .
\end{aligned}
\end{equation}
Then there exists an asymptotically cylindrical Ricci-flat K\"{a}hler form $\omega$ on
$X=\overline{X}\setminus D$ such that
\begin{equation*}
\begin{aligned}
\Omega -\Omega_{\rm cyl}=\der\zeta,\quad
&\omega -\omega_{\rm cyl}=\der\xi\quad\text{for some }\zeta\text{ and }\xi\text{ with}\\
\norm{\nabla_{g_{\rm cyl}}^j\zeta}_{g_{\rm cyl}}=O(e^{-\beta t}),\quad
&\norm{\nabla_{g_{\rm cyl}}^j\xi}_{g_{\rm cyl}}=O(e^{-\beta t})
\quad\text{for all }j\geqslant 0\text{ and }0<\beta <\min\set{1/2,\sqrt{\lambda_1}},
\end{aligned}
\end{equation*}
where $\lambda_1$ is the first eigenvalue of the Laplacian $\Delta_{g_D+\der\theta^2}$
acting on $D\times S^1$ with $g_D$ the metric associated with $\kappa_D$.
\end{theorem}
A pair $(\Omega ,\omega )$ consisting of a holomorphic volume form $\Omega$ and a Ricci-flat K\"{a}hler form $\omega$
on an $m$-dimensional K\"{a}hler manifold
normalized so that
\begin{equation*}
\frac{\omega^m}{m!}=\frac{(\I)^{m^2}}{2^m}\Omega\wedge\overline{\Omega}\;(=\text{the volume form})
\end{equation*}
is called a {\it Calabi-Yau structure}.
The above theorem states that there exists a Calabi-Yau structure $(\Omega ,\omega)$
on $X$ asymptotic to a cylindrical Calabi-Yau structure
$(\Omega_{\rm cyl},\omega_{\rm cyl})$ on $N_{D/\overline{X}}\setminus D$
if we multiply $\Omega$ by some constant.

\subsection{Gluing admissible pairs}\label{sec:gluing}
Hereafter we will only consider admissible pairs $(\overline{X},D)$ with
$\dim_\C\overline{X}=3$. Also, we will denote $N=N_{D/\overline{X}}$ and $X=\overline{X}\setminus D$.

\subsubsection{The gluing condition}

Let $(\overline{X},\omega')$ be a $3$-dimensional compact K\"{a}hler manifold
and $(\overline{X},D)$ be an admissible pair.
We first define a natural torsion-free $G_2$-structure on $X\times S^1$.

It follows from Theorem \ref{thm:TYKH} that there exists a
Calabi-Yau structure $(\Omega ,\omega)$ on $X$ asymptotic to a cylindrical
Calabi-Yau structure $(\Omega_{\rm cyl},\omega_{\rm cyl})$
on $N\setminus D$, which are written as \eqref{eq:TYKH_Omega}
and \eqref{eq:TYKH_CYcyl}.
We define a $G_2$-structure $\varphi$ on $X\times S^1$ by
\begin{equation}\label{eq:phi}
\varphi =\omega\wedge\der\theta' +\Image\Omega ,
\end{equation}
where $\theta'\in \R /2\pi\Z$ is a coordinate on $S^1$.
Similarly, we define a $G_2$-structure $\varphi_{\rm cyl}$
on $(N\setminus D)\times S^1$ by
\begin{equation}\label{eq:phi_cyl}
\varphi_{\rm cyl} =\omega_{\rm cyl}\wedge\der\theta' +\Image\Omega_{\rm cyl}.
\end{equation}
The Hodge duals of $\varphi$ and $\varphi_{\rm cyl}$ are computed as
\begin{equation}\label{eq:Hodge_dual_phi}
\begin{aligned}
*_{g_\varphi}\varphi &=\frac{1}{2}\omega\wedge\omega -\Real\Omega\wedge\der\theta' ,\\
*_{g_{\varphi_{\rm cyl}}}\varphi_{\rm cyl} &=\frac{1}{2}
\omega_{\rm cyl}\wedge\omega_{\rm cyl} -\Real\Omega_{\rm cyl}\wedge\der\theta' .
\end{aligned}
\end{equation}
Then we see easily from Theorem \ref{thm:TYKH} and
equations \eqref{eq:phi}--\eqref{eq:Hodge_dual_phi} that
\begin{equation}\label{eq:difference}
\begin{aligned}
\varphi -\varphi_{\rm cyl}&=\der\xi\wedge\der\theta' +\Image\der\zeta
=\der\eta_1 ,\\
*_{g_\varphi}\varphi -*_{g_{\varphi_{\rm cyl}}}\varphi_{\rm cyl}
&=(\omega +\omega_{\rm cyl})\wedge\der\xi -\Real\der\zeta\wedge\der\theta'
=\der\eta_2,\\
\text{where}\quad\eta_1 &=\xi\wedge\der\theta'+\Image\zeta ,
\quad\eta_2=(\omega +\omega_{\rm cyl})\wedge\xi -\Real\zeta\wedge\der\theta' .
\end{aligned}
\end{equation}
Thus $\varphi$ and $\varphi_{\rm cyl}$ are both torsion-free $G_2$-structures,
and $(X\times S^1, \varphi )$ is asymptotic to
$((N\setminus D)\times S^1,\varphi_{\rm cyl})$.
Note that the cylindrical end of $X\times S^1$ is diffeomorphic to
$(N\setminus D)\times S^1\simeq D\times S^1\times S^1\times\R_+
=\{(x_\alpha ,\theta ,\theta' ,t)\}$.

Next we consider the condition under which we can glue together
$X_1$ and $X_2$ obtained from admissible pairs
$(\overline{X}_1,D_1)$ and $(\overline{X}_2,D_2)$. For gluing $X_1$
and $X_2$ to obtain a manifold with an approximating
$G_2$-structure, we would like $(X_1,\varphi_1 )$ and
$(X_2,\varphi_2)$ to have the same asymptotic model. Thus we put the
following
\begin{description}
\item[\it Gluing condition] There exists a diffeomorphism
$F: D_1\times S^1\times S^1\longrightarrow D_2\times S^1\times S^1$
between the cross-sections of the cylindrical ends such that
\begin{equation}\label{eq:gluing_condition}
F_T^*\varphi_{2,\rm cyl} =\varphi_{1,\rm cyl}\quad\text{for all }T>0,
\end{equation}
where $F_T:D_1\times S^1\times S^1\times (0,2T)\longrightarrow D_2\times S^1\times S^1\times (0,2T)$
is defined by
\begin{equation*}
F_T(x_1,\theta_1, \theta'_1, t)=(F( x_1,\theta_1, \theta'_1 ),2T-t)\quad\text{for }
(x_1,\theta_1, \theta'_1, t)\in D_1\times S^1\times S^1\times (0,2T) .
\end{equation*}
\end{description}
\begin{lemma}\label{lem:gluing_condition}
Suppose that
there exists an isomorphism $f:D_1\longrightarrow D_2$ such that $f^*\kappa_{D_2}=\kappa_{D_1}$.
If we define a diffeomorphism $F$ between the cross-sections of the cylindrical ends by
\begin{equation*}
\xymatrix@R=-1ex{
F_T:D_1\times S^1\times S^1\ar[r]&D_2\times S^1\times S^1.\\
\din &\din\\
(x_1,\theta_1 ,\theta'_1 )\ar@{|->}[r]&(x_2,\theta_2 ,\theta'_2 )
=(f(x_1),-\theta_1 ,\theta'_1 )
}
\end{equation*}
Then the gluing condition \eqref{eq:gluing_condition} holds,
where we change the sign of $\Omega_{2,\rm cyl}$
(and also the sign of $\Omega_2$ correspondingly).
\end{lemma}
\begin{proof}
It follows by a straightforward calculation
using \eqref{eq:TYKH_CYcyl} and \eqref{eq:phi_cyl}.
\end{proof}

\begin{remark}\label{rem:matching_condition}\rm
In the constructions of compact $G_2$-manifolds by Kovalev \cite{K03} and Kovalev-Lee \cite{KL11},
the map $F:D_1\times S^1\times S^1\longrightarrow D_2\times S^1\times S^1$ is defined by
\begin{equation*}
F(x_1,\theta_1 ,\theta'_1 )=(x_2,\theta_2 ,\theta'_2)=(f(x_1),\theta'_1 ,\theta_1)
\quad\text{for }(x_1,\theta_1 ,\theta'_1 )\in D_1\times S^1\times S^1,
\end{equation*}
so that $F$ {\it twists} the two $S^1$ factors.
Then in order for the gluing condition \eqref{eq:gluing_condition} to hold,
the isomorphism $f:D_1\longrightarrow D_2$ between $K3$ surfaces must satisfy
\begin{equation*}
f^*\kappa_2^I=-\kappa_1^J,\quad f^*\kappa_2^J=\kappa_1^I,\quad f^*\kappa_2^K=\kappa_1^K,
\end{equation*}
where $\kappa_i^I,\kappa_i^J,\kappa_i^K$ are defined by
\begin{equation*}
\kappa_{D_i}=\kappa_i^I ,\quad\Omega_{D_i}=\kappa_i^J +\I\kappa_i^K.
\end{equation*}
Instead, Kovalev and Lee put a weaker condition
(which they call the {\it matching condition})
\begin{equation*}
f^*[\kappa_2^I]=-[\kappa_1^J ],\quad f^*[\kappa_2^J ]=[\kappa_1^I ],
\quad f^*[\kappa_2^K ]=[\kappa_1^K] ,
\end{equation*}
which is sufficient for the existence of $f$
by the global Torelli theorem of $K3$ surfaces.
Following Kovalev's argument in \cite{K03}, we can weaken the condition
$f^*\kappa_2=\kappa_1$ in Lemma \ref{lem:gluing_condition} to $f^*[\kappa_2]=[\kappa_1]$.
\end{remark}
\subsubsection{Approximating $G_2$-structures}\label{section:T-approx}
Now we shall glue $X_1\times S^1$ and $X_2\times S^1$
under the gluing condition \eqref{eq:gluing_condition}.
Let $\rho :\R\longrightarrow [0,1]$ denote a cut-off function
\begin{equation*}
\rho (x)=
\begin{cases}
1&\text{if }x\leqslant 0,\\
0&\text{if }x\geqslant 1,
\end{cases}
\end{equation*}
and define $\rho_T :\R\longrightarrow [0,1]$ by
\begin{equation}\label{rho_T}
\rho_T (x)=\rho (x-T+1)=
\begin{cases}
1&\text{if }x\leqslant T-1,\\
0&\text{if }x\geqslant T.
\end{cases}
\end{equation}
Setting an approximating Calabi-Yau structure $(\Omega_{i,T}, \omega_{i,T})$ by
\begin{equation*}
\Omega_{i,T}=
\begin{cases}
\Omega_i -\der (1-\rho_{T-1})\zeta_i &\text{on }\{ t\leqslant T-1\} ,\\
\Omega_{i,\rm cyl}+\der\rho_{T-1}\zeta_i &\text{on }\{ t\geqslant T-2\}
\end{cases}
\end{equation*}
and similarly
\begin{equation*}
\omega_{i,T}=
\begin{cases}
\omega_i -\der (1-\rho_{T-1})\xi_i &\text{on }\{ t\leqslant T-1\} ,\\
\omega_{i,\rm cyl}+\der\rho_{T-1}\xi_i &\text{on }\{ t\geqslant T-2\},
\end{cases}
\end{equation*}
we can define a $\der$-closed (but not necessarily $\der^*$-closed) $G_2$-structure
$\varphi_{i,T}$ on each $X_i\times S^1$ by
\begin{equation*}
\varphi_{i,T}=\omega_{i,T}\wedge\der\theta'_i +\Image\Omega_T .
\end{equation*}
Note that $\varphi_{i,T}$ satisfies
\begin{equation*}
\varphi_{i,T}=
\begin{cases}
\varphi_i&\text{on }\{ t<T-2\} ,\\
\varphi_{i,\rm cyl}&\text{on }\{ t>T-1\}
\end{cases}
\end{equation*}
and that
\begin{equation}\label{eq:phi_T-phi_cyl}
\norm{\varphi_{i,T} -\varphi_{i,\rm cyl}}_{g_{\varphi_{i,\rm cyl}}}=O(e^{-\beta T})
\quad\text{for all }0<\beta <\min\set{1/2,\sqrt{\lambda_1}}.
\end{equation}

Let $X_{1,T}=\{ t_1<T+1\}\subset X_1$ and $X_{2,T}=\{ t_2<T+1\}\subset X_2$.
We glue $X_{1,T}\times S^1$ and $X_{2,T}\times S^1$
along $D_1\times S^1\times\{ T-1<t_1<T+1\}\times S^1\subset X_{1,T}\times S^1$
and $D_2\times S^1\times\{ T-1<t_2<T+1\}\times S^1\subset X_{2,T}\times S^1$
to construct a compact $7$-manifold $M_T\times S^1$
using the gluing map $F_T$
(more precisely, $\widetilde{F}_T=(\Phi_2,\id_{S^1} )\circ F_T\circ (\Phi_1^{-1},\id_{S^1} )$,
where $\Phi_1$ and $\Phi_2$
are the diffeomorphisms given in Lemma \ref{lem:tub.nbd.thm}).
Also, we can glue together $\varphi_{1,T}$ and $\varphi_{2,T}$ to obtain a $3$-form
$\varphi_T$ on $M_T$.
It follows from Lemma \ref{lem:ep-nbd_2} and \eqref{eq:phi_T-phi_cyl} that there exists $T_*>0$
such that $\varphi_T\in\mathcal{P}^3(M_T\times S^1)$ for all $T$ with $T>T_*$,
so that the Hodge star operator $*=*_{g_{\varphi_T}}$ is well-defined.
Thus we can define a $3$-form $\psi_T$ on $M_T\times S^1$ with
$\der^*\varphi_T=\der^*\psi_T$ by
\begin{equation}\label{eq:psi_T}
*\psi_T=*\varphi_T-\left(\frac{1}{2}\omega_T\wedge\omega_T
-\Real\Omega_T\wedge\der\theta'\right) .
\end{equation}

\begin{proposition}\label{prop:estimates}
There exist constants $A_{p,k,\beta}$
independent of $T$ such that for $\beta\in (0,\set{1/2,\sqrt{\lambda_1}})$ we have
\begin{equation*}
\Norm{\psi_T}_{L^p_k}\leqslant A_{p,k,\beta}\, e^{-\beta T},
\end{equation*}
where all norms are measured using $g_{\varphi_T}$.
\end{proposition}
\begin{proof}
These estimates follow in a straightforward way from Theorem \ref{thm:TYKH}
and equation \eqref{eq:difference}
by arguments similar to those in \cite{D09}, Section 3.5.
\end{proof}

\subsection{Gluing construction of Calabi-Yau threefolds}
Here we give the main theorems for constructing Calabi-Yau threefolds.
\begin{theorem}\label{thm:main}
Let $(\overline{X}_1,\omega'_1)$ and $(\overline{X}_2,\omega'_2)$ be
compact K\"{a}hler manifold with $\dim_\C\overline{X}_i=3$
such that $(\overline{X}_1,D_1)$ and $(\overline{X}_2,D_2)$ are admissible pairs.
Suppose there exists an isomorphism $f: D_1\longrightarrow D_2$ such that
$f^*\kappa_2=\kappa_1$,
where $\kappa_i$ is the unique Ricci-flat K\"{a}hler form on $D_i$ in the K\"{a}hler class
$[\restrict{\omega'_i}{D_i}]$.
Then we can glue toghether $X_1$ and $X_2$ along their cylindrical ends to obtain
a compact manifold $M$. The manifold $M$ is a Calabi-Yau threefold, i.e.,
$b^1(M)=0$ and $M$ admits a Ricci-flat K\"{a}hler metric.
\end{theorem}
\begin{corollary}\label{cor:doubling}
Let $(\overline{X},D)$ be an admissible pair with $\dim_\C\overline{X}=3$.
Then we can glue two copies of $X$ along their cylindrical ends to obtain a compact manifold $M$.
The manifold $M$ is a Calabi-Yau threefold.
\end{corollary}
\begin{remark}\rm
As stated in Remark \ref{rem:matching_condition}, 
the condition $f^*\kappa_2=\kappa_1$ in Theorem \ref{thm:main}
can be weakened to $f^*[\kappa_2]=[\kappa_1]$ using Kovalev's argument in \cite{K03}.
But we don't go into details here because we don't need the weaker condition 
for getting Corollary \ref{cor:doubling} from Theorem \ref{thm:main}.
\end{remark}

\begin{proof}[Proof of Theorem $\ref{thm:main}$]
We shall prove the existence of a torsion-free $G_2$-structure on $M_T\times S^1$
constructed in Section \ref{sec:gluing} for sufficiently large $T$. Then $M=M_T$ will be
the desired Calabi-Yau threefold according to the following

\begin{lemma}\label{lem:RFonM}
If $M\times S^1$ admits a torsion-free $G_2$-structure,
then $M$ admits a Ricci-flat K\"{a}hler metric.
\end{lemma}
\begin{proof}
Since both $X_1$ and $X_2$ are simply-connected by Definition
\ref{def:admissible} (d),
the resulting manifold $M=M_T$ is also simply-connected. 
Let us consider a Riemannian metric on $M\times S^1$ with holonomy contained in $G_2$,
which is induced by a torsion-free $G_2$-structure.
Then by the Cheeger-Gromoll splitting theorem (see e.g. Besse \cite{Bs87}, Corollary 6.67), 
the universal Riemannian covering of $M\times S^1$ is isometric to a product 
Riemannian manifold $N\times\R^q$ with holonomy contained in $G_2$ for some $q$,
where $N$ is a simply-connected $(7-q)$-manifold and $\R^q$ has a flat metric.
Meanwhile, the natural map $M\times\R\longrightarrow M\times S^1$ is also the universal covering.
By the uniqueness of the universal covering, we have a diffeomorphism $\phi :M\times\R\longrightarrow N\times\R^q$, 
so that $q=1$ and $N$ is $6$-dimensional.
Since the flat metric on $\R$ does not contribute to the holonomy of $N\times\R$,
$N$ itself has holonomy contained in $G_2$. But the holonomy group of a simply-connected
Riemannian $6$-manifold is at most $\SO (6)$, and so it must be
contained in $\SO (6)\cap G_2 =\SU (3)$. Thus $N$ admits a Ricci-flat K\"{a}hler metric.

Now we shall prove that $N$ is indeed diffeomorphic to $M$.
For this purpose, we use the classification of closed, oriented simply-connected $6$-manifolds
by Wall, Jupp and Zhubr 
(see the website of the Manifold Atlas Project, $6$-manifolds: 1-connected \cite{MA13}
for a good overview which includes further references).
Then we see that $M$ and $N$ are diffeomorphic
if there is an isomorphism between the cohomology rings $H^*(M)$ and $H^*(N)$
preserving the second Stiefel-Whitney classes $w_2$ and the first Pontrjagin classes $p_1$ 
(the rest of the invariants are completely determined by the cohomology rings).
Such a ring isomorphism is induced by the diffeomorphism 
$\phi :M\times\R\longrightarrow N\times\R$ via the composition
\begin{equation*}
H^*(N)\cong H^*(N\times\R )\stackrel{\phi^*}{\cong}H^*(M\times\R )\cong H^*(M).
\end{equation*}
This proves that $N$ is diffeomorphic to $M$, and hence $M$ admits a Ricci-flat K\"{a}hler metric.
\end{proof}
Now it remains to prove the existence of a torsion-free $G_2$-structure on $M_T\times S^1$
for sufficiently large $T$. 
We recall the following result which reduces the existence of a torsion-free $G_2$-structure 
to the sovlability of a nonlinear partial differential equation.
\begin{theorem}[Joyce \cite{J00}, Theorem 10.3.7]\label{thm:Joyce}
Let $\varphi$ be a $G_2$-structure on a comact $7$-manifold $M'$ with $\der\varphi =0$. 
Suppose $\eta$ is a $2$-form on $M'$ with $\Norm{\der\eta}_{C^0}\leqslant \epsilon_1$,
and $\psi$ is a $3$-form on $M'$ with $\der^*\psi =\der^*\varphi$ and $\Norm{\psi}_{C^0}\leqslant\epsilon_1$,
where $\epsilon_1$ is a constant independent of the $7$-manifold $M'$ with $\epsilon_1\leqslant\rho_*$.
Let $\eta$ satisfy the nonlinear elliptic partial differential equation
\begin{equation}\label{G_2-PDE}
(\der\der^* +\der^*\der)\eta =\der^*\left( 1+ \frac{1}{3}\braket{\der\eta ,\varphi}_{g_\varphi}\right)\psi +*\der F(\der\eta ).
\end{equation}
Here $F$ is a smooth function from the closed ball of radius $\epsilon_1$ in $\wedge^3 T^*M'$
to $\wedge^4 T^*M'$ with $F(0)=0$, and if $\chi ,\xi\in C^\infty(\wedge^3T^*M')$ and $\norm{\chi},\norm{\xi}\leqslant\epsilon_1$, 
then we have the quadratic estimates 
\begin{equation*}
\begin{aligned}
\norm{F(\chi )-F(\xi )}&\leqslant\epsilon_2 \norm{\chi -\xi}(\norm{\chi}+\norm{\xi}),\\
\norm{\der (F(\chi )-F(\xi ))}&\leqslant\epsilon_3
\left\{\norm{\chi -\xi}(\norm{\chi}+\norm{\xi})\norm{\der^*\varphi}+\norm{\nabla (\chi -\xi)}(\norm{\chi}+\norm{\xi})
+\norm{\chi -\xi}(\norm{\nabla\chi}+\norm{\nabla\xi})\right\}
\end{aligned}
\end{equation*}
for some constants $\epsilon_2, \epsilon_3$ independent of $M'$.
Then $\widetilde{\varphi}=\varphi +\der\eta$ is a torsion-free $G_2$-structure on $M'$.
\end{theorem}
To solve \eqref{G_2-PDE} in our construction, 
we use the following gluing theorem based on the analysis of Kovalev and Singer \cite{KS01}.
\begin{theorem}[Kovalev \cite{K03}, Theorem 5.34]\label{thm:Kovalev}
Let $\varphi=\varphi_T, \psi=\psi_T$ and $M'=M_T\times S^1$ be 
as constructed in Section $\ref{section:T-approx}$, with $\der^*\psi_T=\der^*\varphi_T$ and
the estimates in Proposition $\ref{prop:estimates}$.
Then there exists $T_0>0$ such that the following is true.

For each $T>T_0$, there exists a unique smooth $2$-form $\eta_T$ on $M_T\times S^1$
with $\Norm{\eta_T}_{L^p_2}\leqslant B_{p,\beta}e^{-\beta T}$ and $\Norm{\eta_T}_{C^1}\leqslant C_\beta e^{-\beta T}$ for any 
$\beta\in (0,\max\set{1/2,\sqrt{\lambda_1}})$ such that $\eta =\eta_T$ satisfies equation $\eqref{G_2-PDE}$, 
where $B_{p,\beta}$ and $C_\beta$ are independent of $T$. 
\end{theorem}
\begin{proof}
The assertion is proved in \cite{K03} when $d(\overline{X}_1)=0$ or $d(\overline{X}_2)=0$, 
where $d(\overline{X}_j)$ is the dimension of the kernel of $\iota_j :H^2(X_j,\R )\longrightarrow H^2(D_j,\R )$ 
defined in Section 4.
This condition applies to admissible pairs of Fano type, but not to ones of non-symplectic type
(see also the proof of Proposition 5.38 in \cite{K03} and the remarks after Lemma 2.6 in \cite{KL11}, p.199).
However, the above theorem is still valid in the non-symplectic case, 
by a direct application of Kovalev-Singer \cite{KS01}, Proposition 4.2.
\end{proof}
Applying Theorem \ref{thm:Kovalev} to Theorem \ref{thm:Joyce}, 
we see that $\widetilde{\varphi}_T=\varphi_T +\der\eta_T$ yields 
a torsion-free $G_2$-structure on $M_T\times S^1$ for sufficiently large $T$.
Combined with Lemma \ref{lem:RFonM}, this completes the proof of Theorem \ref{thm:main}.
\end{proof}
\begin{remark}\rm
In the proof of Theorem \ref{thm:main}, to solve equation \eqref{G_2-PDE} given in Theorem \ref{thm:Joyce}
we may also use Joyce's book \cite{J00}, Theorem 11.6.1, 
where we need uniform bounds of the injectivity radius and Riemann curvature of $M_T\times S^1$
from below and above respectively. Obviously, we have such bounds 
because $X_1$ and $X_2$ are cylindrical manifolds with an asymptotically cylindrical metric.
\end{remark}

\section{Betti numbers of the resulting Calabi-Yau threefolds}\label{sec:topological inv}
We shall compute Betti numbers of the Calabi-Yau threefolds $M$
obtained in the doubling construction given in Corollary
\ref{cor:doubling}. Also, we shall see that the Betti numbers of $M$
are completely determined by those of the compact K\"ahler threefolds
$\overline{X}$.

In our doubling construction, we take two copies $(\overline{X}_j,D_j)$ of an admissible pair $(\overline{X},D)$ for $j=1,2$. 
Let $X_j=\overline{X}_j\setminus D_j$. We consider a homomorphism
\begin{equation}\label{map:iota}
 \iota_j: H^2(X_j, \R) \longrightarrow H^2(D_j\times S^1, \R)
\stackrel{\cong}{\longrightarrow} H^2(D_j,\R ),
  \end{equation}
where the first map is induced by the embedding $D_j\times S^1
\longrightarrow {X}_j$ and the second comes from the K\"{u}nneth
theorem. Set $d=d_j=d(\overline{X}_j)=\dim_\R \Ker \iota_j$. It is readily seen that
\begin{equation}\label{eq:dim iota}
\dim_\R \Image \iota_j =b^2(X)-d.
\end{equation}
The following formula seems to be well-known for compact K\"ahler threefolds (see \cite{L10}, Corollary $8.2$).
\begin{proposition}\label{prop:betti}
Let $(\overline{X}_j,D_j)$ be two copies of an admissible pair
$(\overline{X},D)$ for $j=1,2$ and let $d$ be as above. Then the
Calabi-Yau threefold $M$ obtained by the doubling construction in
Corollary $\ref{cor:doubling}$ has Betti numbers
\begin{equation}\label{eq:betti}
\begin{aligned}
\begin{cases}
b^1(M)=0,\\
b^2(M)=b^2(\overline{X})+d, \\
b^3(M)=2\left (b^3(\overline{X})+23+d-b^2(\overline{X})\right ).
\end{cases}
\end{aligned}
\end{equation}
Also, the Euler characteristic $\chi (M)$ is given by
\begin{align*}
\chi(M)=2(\chi(\overline{X})-\chi(D)).
\end{align*}
\end{proposition}
\begin{proof}
Obviously, the second statement holds for our construction. Now we
restrict ourselves to find the second and third Betti numbers of
$M$ because $M$ is simply-connected. Since the normal bundle
$N_{D_{j}/\overline{X}_j}$  
is trivial in our assumption, there is a tubular neighborhood $U_j$ of $D_j$ in
$\overline{X}_j$ such that
\begin{equation}\label{eq:tub nbd1}
\overline{X}_j=X_j\cup U_j \qquad \text{and}\qquad X_j\cap U_j\simeq
D_j\times S^1\times \R_{>0} .
\end{equation}
Up to a homotopy equivalence, $X_j \cap U_j \sim D_j\times S^1$ as
$U_j$ contracts to $D_j$. Applying the Mayer-Vietoris theorem to
\eqref{eq:tub nbd1}, we see that
\begin{equation}\label{eq:KovLee}
b^2(\overline{X})=b^2(X)+1 \qquad \text{and}\qquad
b^3(X)=b^3(\overline{X})+22+d-b^2(X)
\end{equation}
(see \cite{KL11}, $(2.10)$). 
We next consider homotopy equivalences
\begin{equation}\label{eq:tub nbd2}
M \sim X_1 \cup X_2, \qquad X_1\cap X_2 \sim D\times S^1.
\end{equation}
Again, let us apply the Mayer-Vietoris theorem to \eqref{eq:tub nbd2}. 
Then we obtain the long exact sequence
\begin{equation}\label{seq:MV3}
\xymatrix{ 0\ar[r] &H^0(D) \ar[r]^-{\delta^1}&
H^2(M)\ar[r]^-{\alpha^2}& H^2(X_1)\oplus H^2(X_2) \ar[r]^-{\beta^2}
&H^2(D) \ar[r]&\cdots .}
\end{equation}
Note that the map $\beta^2$ in \eqref{seq:MV3} is given by
\begin{equation*}
\iota_1+f^*\iota_2: H^2(X_1,\R)\oplus H^2(X_2,\R ) \longrightarrow
H^2(D,\R ),
\end{equation*}
where
\begin{equation*}
\iota_j: H^2(X_j, \R) \longrightarrow H^2(D_j, \R) 
\end{equation*}
are homomorphisms defined in \eqref{map:iota} and
\[
f^* : H^2(D_2,\R) \longrightarrow H^2(D_1,\R)
\]
is the pullback of the identity
$f:D_1\stackrel{\cong}{\longrightarrow}D_2$. Hence we see from
\eqref{eq:dim iota} that
\[
\dim_\R \Image (\iota_1+f^*\iota_2)=b^2(X)-d.
\]
This yields
 \begin{align*}
 b^2(M)&=\dim_\R \Ker \alpha^2 +\dim_\R \Image \alpha^2 \\
 &=\dim_\R \Image \delta^1+\dim_\R \Ker (\iota_1+f^*\iota_2) \\
 &=1+2b^2(X)-(b^2(X)-d)=b^2(\overline{X})+d,
 \end{align*}
where we used \eqref{eq:KovLee} for the last equality. Remark that
$b^2(X_1)=b^2(X_2)$ holds for our computation. To find $b^3(M)$, we
shall consider a homomorphism
\begin{equation}\label{map:tau}
\tau_j: H^3(X_j,\R) \longrightarrow H^2(D_j, \R)
\end{equation}
which is induced by the embedding $U_j \cap X_j \longrightarrow X_j$
combined with
\begin{equation*}
X_j\cap U_j \simeq D_j\times S^1 \times \R_{>0} \qquad \text{and}
\qquad H^3(D_j\times S^1, \R) \cong H^2(D_j, \R).
\end{equation*}
The reader should be aware of the following lemma.
\begin{lemma}[Kovalev-Lee \cite{KL11}, Lemma 2.6]\label{lem:KovLee}
Let $\iota_j$ and $\tau_j$ be homomorphisms defined in
\eqref{map:iota} and \eqref{map:tau} respectively. Then we have the
orthogonal decomposition
\[
H^2(D_j,\R)=\Image \tau_j \oplus \Image \iota_j
\]
with respect to the intersection form on $H^2(D_j,\R)$ for each
$j=1,2$.
\end{lemma}
In an analogous way to the computation of $b^2(M)$, we apply the
Mayer-Vietoris theorem to \eqref{eq:tub nbd2}:
\begin{equation}\label{seq:MV4}
\xymatrix@R=-1ex{
\cdots\ar[r]& H^2(X_1)\oplus H^2(X_2)\ar[r]^-{\iota_1+f^*\iota_2}
& H^2(D)\ar[r]^{\delta^2} & H^3(M)\ar[r]&\\
\ar[r]^-{\alpha^3}& H^3(X_1)\oplus H^3(X_2) \ar[r]^-{\beta^3}
&H^2(D)\ar[r]& \cdots .& }
\end{equation}
Similarly, the map $\beta^3$ is given by
\[
\tau_1+f^*\tau_2: H^3(X_1)\oplus H^3(X_2)\longrightarrow H^2(D) .
\]
On one hand, Lemma $\ref{lem:KovLee}$ and \eqref{eq:dim iota} show
that
\[
\dim_\R \Image \tau_j =22+d-b^2(X).
\]
Hence we find that
\begin{align}
\begin{split}\label{eq:dim tau}
\dim_\R \Ker(\tau_1+f^*\tau_2)&=b^3(X_1)+b^3(X_2)-\dim_\R \Image (\tau_1+f^*\tau_2)\\
&= 2b^3(X)-(22+d-b^2(X)).
\end{split}
\end{align}
On the other hand, we have the equality
\[
22=\dim_\R \Image \delta^2 +\dim_\R\Image (\iota_1+f^*\iota_2)
\]
by combining the well-known result on the cohomology of a $K3$ surface $D$
with the Mayer-Vietoris long exact sequence \eqref{seq:MV4}. Then we have 
\begin{equation}\label{eq:dim alpha3}
\dim_\R \Ker \alpha^3 =\dim_\R \Image \delta^2=22-b^2(X)+d.
\end{equation}
Thus we find from \eqref{eq:dim tau} and \eqref{eq:dim alpha3} that
 \begin{align*}
 b^3(M)&=\dim_\R \Ker \alpha^3 +\dim_\R \Ker (\tau_1+f^*\tau_2)=2b^3(X).
\end{align*}
Substituting the above equation into \eqref{eq:KovLee}, we obtain the assertion.
\end{proof}
\begin{remark}\rm
This formula shows that the topology of the resulting Calabi-Yau
threefolds $M$ only depends on the topology of the given compact
K\"ahler threefolds $\overline{X}$. Also one can determine the Hodge
diamond of $M$ from Proposition \ref{prop:betti} because we already
know that $h^{0,0}=h^{3,0}=1$ and $h^{1,0}=h^{2,0}=0$ by the well-known
result on Calabi-Yau manifolds (see \cite{J00}, Proposition 6.2.6).
\end{remark}

\section{Two types of admissible pairs}
In this section, we will see the construction of admissible pairs
$(\overline{X},D)$ which will be needed for obtaining Calabi-Yau
threefolds in the doubling construction. There are two types of
admissible pairs. One is said to be \textit{of Fano type}, and the
other \textit{of non-symplectic type}. We will give explicit
formulas for topological invariants of the resulting Calabi-Yau
threefolds from these two types of admissible pairs. For the
definition of admissible pairs, see Definition \ref{def:admissible}.

\subsection{Fano type}
Admissible pairs $(\overline{X},D)$ are ingredients in our
construction of Calabi-Yau threefolds and then it is important how
to explore appropriate compact K\"ahler threefolds $\overline{X}$
with an anticanonical $K3$ divisor $D\in\norm{-K_{\overline{X}}}$. In
\cite{K03}, Kovalev constructed such pairs from nonsingular Fano
varieties.

\begin{theorem}[Kovalev \cite{K03}]\label{th:ad Fano}
Let $V$ be a Fano threefold, $D\in\norm{-K_V}$ a $K3$ surface, and let $C$ be a
smooth curve in $D$ representing the self-intersection class of $D\cdot D$.
Let $\varpi : \overline{X}\dasharrow V$ be the blow-up of $V$ along
the curve $C$. Taking the proper transform of $D$ under the blow-up
$\varpi$, we still denote it by $D$. Then $(\overline{X},D)$ is an
admissible pair.
\end{theorem}
\begin{proof}
See \cite{K03}, Corollary $6.43$, and also Proposition $6.42$.
\end{proof}

An admissible pair $(\overline{X},D)$ given in Theorem $\ref{th:ad
Fano}$ is said to be of \textit{Fano type} because this pair arises
from a Fano threefold $V$. Note that $\overline{X}$ itself is {\it
not} a Fano threefold in this construction.

\begin{proposition}\label{prop:betti Fano type}
Let $V$ be a Fano threefold and $(\overline{X},D)$ an admissible
pair of Fano type given in Theorem $\ref{th:ad Fano}$. Let $M$ be
the Calabi-Yau threefold constructed from two copies of
$(\overline{X},D)$ by Corollary $\ref{cor:doubling}$. Then we have
\begin{equation*}
\begin{cases}
b^2(M)=b^2(V)+1,\\
b^3(M)=2\left(b^3(V)-K_V^3+24-b^2(V)\right).
\end{cases}
\end{equation*}
In particular, the cohomology of $M$ is completely determined by the
cohomology of $V$.
\end{proposition}
\begin{proof}
Let $d$ be the dimension of the kernel of the homomorphism
\[
\iota : H^2(X,\R) \longrightarrow H^2(D,\R)
\]
as in Section $\ref{sec:topological inv}$. Then note that $d=0$ by
the Lefschetz hyperplane theorem whenever $(\overline{X}, D)$ is of
Fano type. Applying the well-known result on the cohomology of blow-ups,
one can find that
\[
H^2(\overline{X})\cong H^2(V)\oplus \R \qquad \text{and} \qquad
H^3(\overline{X}) \cong H^3 (V) \oplus \R^{2g(V)},
\]
where $\displaystyle g(V)=\frac{-K_V^3}{2}+1$ is the genus of a Fano
threefold (see \cite{K03}, (8.52)). This yields
\begin{equation*}\label{eq:betti Fano}
b^2(\overline{X})=b^2(V)+1 \qquad \text{and} \qquad
b^3(\overline{X})=b^3(V)+2g(V).
\end{equation*}
Substituting this into Proposition $\ref{prop:betti}$, we can show
our result.
\end{proof}

\begin{remark}\rm
We have another method to compute the Euler characteristic
$\chi(M)$. In fact, we can see easily that if $\overline{X}$ is the
blow-up of $D$ along $C$ then the Euler characteristic of
$\overline{X}$ is given by
\[
\chi(\overline{X})=\chi(V)-\chi(C)+\chi(E)
\]
where $E$ is the exceptional divisor of the blow-up $\varpi$. Hence
we can independently compute $\chi(M)$ by
\begin{align*}
\chi(M)&=2(\chi(\overline{X})-\chi(D)) \\
&=2(\chi(V)+\chi(C)-\chi(D))
\end{align*}
because $E$ is a $\C P^1$-bundle over the smooth curve $C$. Since the
Euler characteristic is also given by $\chi(M)=\sum_{i=0}^{\dim_\R
M}(-1)^ib^i(M)$, we can check the consistency of our computations.
\end{remark}

\subsection{Non-symplectic type}\label{sec:non-symplectic}

In \cite{KL11}, Kovalev and Lee gave a large class of admissible
pairs $(\overline{X},D)$ from $K3$ surface $S$ with a non-symplectic
involution $\rho$. They also used the classification result of $K3$
surfaces $(S,\rho)$  due to Nikulin \cite{N79,N80,N83} for obtaining
new examples of compact irreducible $G_2$-manifolds. Next we will give a quick review on this construction. 
For more details, see \cite{KL11} Section $4$. 

\subsubsection{{$K3$ surfaces with a non-symplectic involution}}
Let $S$ be a $K3$ surface. Then the vector space $H^{2,0}(S)$ is
spanned by a holomorphic volume form $\Omega$, which is unique up to
multiplication of a constant. An automorphism $\rho$ of $S$ is said
to be \textit{non-symplectic} if its action on $H^{2,0}(S)$ is
nontrivial. We shall consider a non-symplectic involution: 
\[
\rho^2=\mathrm{id} \qquad \text{and} \qquad \rho^*\Omega
= - \;\Omega . 
\]
The intersection form of $S$ associates a lattice structure, i.e., a
free abelian group of finite rank endowed with a nondegenerate
integral bilinear form which is symmetric. We refer to this lattice
as the \textit{$K3$ lattice}. It is crucial that the $K3$ lattice
has a nice property for a geometrical description of $S$. Hence we
shall review some fundamental concepts of lattice theory which will
be needed later.

Recall that the lattice $L$ is said to be \textit{hyperbolic} if the
signature of $L$ is $(1,t)$ with $t>0$. In particular, we are
interested in the case where $L$ is \textit{even}, i.e., the
quadratic form $x^2$ is $2\Z$-valued for any $x\in L$. We
can regard $L$ as a sublattice of $L^*=\mathrm{Hom} (L,\Z )$
by considering the canonical embedding $i:L\longrightarrow L^*$
given by $i(x)y=\braket{x,y}$ for $y\in L^*$. Then $L$ is said to be
\textit{unimodular} if the quotient group $L^*/L$ is trivial.
In general, $L^*/L$ is a finite abelian group and is called the
\textit{discriminant group} of $L$. One can see that the cohomology
group $H^2(S,\Z )$ of each $K3$ surface $S$ is a unimodular,
nondegenerate, even lattice with signature $(3,19)$. Let $H$ and
$E_8$ denote the hyperbolic plane lattice 
$\begin{pmatrix}
0&1
\\1&0\end{pmatrix}$ 
and the root lattice associated to the root system $E_8$ respectively. 
Then $H^2(S,\Z )$ is isomorphic to
$3H\oplus2(-E_8)$. Let us choose a \textit{marking}
$\phi:H^2(S,\Z )\longrightarrow L$ of $S$, that is, a lattice
isomorphism. It is clear that the pullback $\rho^*$ induces an
isometry of $L$ with order $2$ defined by $\phi\circ
\rho^*\circ \phi^{-1}$. Hence we can consider the
\textit{invariant sublattice} $L^{\rho}$. Then $L$ is said to be
\textit{$2$-elementary} if the discriminant group of $L^{\rho}$ is
isomorphic to ${(\Z_2)}^a$ for some $a\in \Z_{\geqslant 0}$.
\begin{theorem}[Nikulin \cite{N79, N80, N83}]
Let $(S,\rho)$ be a $K3$ surface $S$ with a non-symplectic involution $\rho$. 
Then the deformation class of $(S,\rho)$ depends only on the following triplet $(r,a,\delta)\in
\Z^3$ given by
\begin{enumerate}
\item[(i)] $r=\rank\; L^{\rho}$,
\item[(ii)] $(L^{\rho})^*/L^{\rho}\cong {(\Z_2)}^a$, and
\item[(iii)] $\delta(L^{\rho})=
\begin{cases}
0 & \text{if\;\;} y^2\in\Z \text{\;\;for all\;\;} y\in (L^{\rho})^*, \\
1 & \text{otherwise}  \: .
\end{cases}$
\end{enumerate}
\end{theorem}

\subsubsection{The cohomology for non-symplectic type}
Let $\sigma$ be a holomorphic involution of $\C P^1$ given by
\[
\sigma: \C P^1 \longrightarrow \C P^1, \qquad z \longmapsto - z
\]
in the standard local coordinates. Let $G$ be the cyclic group of
order $2$ generated by $\rho\times\sigma$. Let $X'$ be the trivial
$\C P^1$-bundle over $S$. Then the group $G$ naturally acts on $X'$.
Taking a point $x$ in the fixed locus $W=(X')^G$ under the action of
$G$, we denote the stabilizer of $x$ as $G_x$. Then $G_x$ is an
endomorphism of the tangent space $T_x X'$ which satisfies $G_x
\subset {\rm{SL}}(T_x X')$. Define the quotient variety
\[
Z=X'/G_x
\]
and then the above condition $G_x \subset \mathrm{SL}(T_xX')$ yields
that the algebraic variety $Z$ admits only Gorenstein quotient
singularities \cite{W74}. Therefore, there is a crepant resolution
$\overline{\pi}: \overline{X} \dasharrow Z$ due to Roan's result
(see \cite{R94}, Main theorem).

Let $W$ be the fixed locus of $X'$ under the action of $G$
as above. We assume that $W$ is nonempty. In fact, this condition
always holds unless $(r,a,\delta)=(10,10,0)$, i.e., $S/\rho$ is an Enriques surface. Then it is known that $W$ is the
disjoint union of some rational curves. Let $\widetilde{\pi}:
\widetilde{X}\dasharrow X'$ be the blow-up of $X'=S\times \C P^1$
along the fixed locus $W$. Then $\widetilde{X}$ is simply-connected
as $X'$ is simply-connected. Also, the action of $G$ on $X'$ lifts
to the action of $\widetilde{G}$ on $\widetilde{X}$ as follows.
Since we have the isomorphism
\[
\widetilde{X}\setminus \widetilde{\pi}^{-1}(W)\cong X' \setminus W,
\]
it suffices to consider the action of $\widetilde{G}$ on a point $x\in
\widetilde{\pi}^{-1}(W)$. Setting $g\cdot x=x$ for all $g\in
\widetilde{G}$ and $x\in \widetilde{\pi}^{-1}(W)$, we have the lift
$\widetilde G$ on $\widetilde X$. Observe that $\widetilde
X/\widetilde G \cong\overline X$ as the quotient of the variety
$\widetilde X$ by $\widetilde G$. Summing up these arguments, we
have the following commutative diagram:
\[
\xymatrix{
 \widetilde{G}\quad\stackrel{\text{lift}}{\curvearrowright}\widetilde{X}
\ar@<3.5ex>@{-->}[d]_{\widetilde{\pi}}\ar@{>>}[r]^-{\widetilde f}
& \overline{X}\ar@{-->}[d]^{\overline{\pi}: \text{ crepant}} \\
G  \quad \curvearrowright X' \ar@{>>}[r]^-{f} & Z }
\]
where $\widetilde f $ (resp. $f$) is the quotient map with respect
to $\widetilde G$ (resp. $G$). Taking a non-fixed point $z\in \C P^1\setminus \set{0,\infty}$, let us define $D'=S\times \set{z}$,
which is a $K3$ divisor on $X'$. Setting $D$ as the image of $D'$ in $Z$, we still denote by $D$ the proper transform of $D$ under $\overline{\pi}$. 
Then we can see that $D$ is isomorphic
to $S$. Furthermore, the normal bundle $N_{D/ \overline{X}}$ is holomorphically trivial. In order to show
$(\overline{X}, D)$ is an admissible pair, we need the following three lemmas due to Kovalev and Lee \cite{KL11}.

\begin{lemma}[Kovalev-Lee \cite{KL11}, Proposition $4.1$]\label{lem:KL1}
$\overline X$ is a compact K\"ahler threefold. Moreover, there
exists a K\"ahler class $[\omega]\in H^2(\overline{X}, \R)$ such
that
\[
[\kappa]=[\restrict{\omega}{D}]\in H^2(D,\R)
\]
where $[\kappa]$ is a $\rho$-invariant K\"ahler class on $D$.
\end{lemma}

\begin{lemma}[Kovalev-Lee \cite{KL11}, Lemma $4.2$]\label{lem:KL2}
$\overline{X}$ and $X=\overline{X}\setminus D$ are
simply-connected whenever $(r,a,\delta)\neq (10,10,0)$. 
\end{lemma}
Although the following lemma is also stated in \cite{KL11}, p.$202$ without a proof, we will prove it here for the reader's convenience.
\begin{lemma}\label{lem:anticano}
$D$ is an anticanonical divisor on $\overline X$.
\end{lemma}
\begin{proof}
To begin with, we consider the divisor $D'= S\times \set{z}$ on $X'=S\times \C P^1$, where $z\in \C P^1\setminus \set{0, \infty}$. Let $p_1: X'\longrightarrow S$ and $p_2: X'\longrightarrow \C P^1$ be the canonical projections. Then we have the isomorphism
\[
K_{X'}\cong p_1^*K_S\otimes p_2^*K_{\C P^1}\cong p_2^*\mathcal O_{\C P^1}(-2),
\]
where we used $K_S\cong \mathcal O_S$ for the second isomorphism. Similarly, we conclude that
\[
[D']\cong p_2^*[z]\cong p_2^*\mathcal O_{\C P^1}(1).
\] 
This yields 
\begin{equation*}\label{eq:div1}
K_{X'}\otimes [2D']\cong \mathcal O_{X'}
\end{equation*}
and hence $c_1(K_{X'}\otimes [2D'])=0$. Since $H^2(Z,\Z )$ is the $G$-invariant part of $H^2(X',\Z )$, the pullback map
$f^*: H^2(Z,\Z )\longrightarrow H^2(X',\Z )$ is injective. Thus,
\[
f^* c_1(K_Z\otimes [D])=c_1(K_{X'}\otimes [2D'])=0
\]
implies $c_1(K_Z\otimes [D])=0$. We remark that
\begin{equation}\label{eq:div2}
D\cap {\rm Sing}(Z)=\emptyset
\end{equation}
because $z\in \C P^1$ is a non-fixed point of $\sigma$. Since $\overline {\pi}$ is a crepant resolution, we have
\[
\overline{\pi}^* K_Z\cong K_{\overline {X}} \qquad \text{and} \qquad \overline{\pi}^*[D] \cong [D]
\]
by (\ref{eq:div2}). Hence $c_1(K_Z\otimes [D])=0$ implies 
\[
c_1(K_{\overline{X}}\otimes [D])=c_1(\overline{\pi}^* K_Z\otimes \overline{\pi}^* [D])=\overline{\pi}^* c_1(K_Z\otimes [D])=0.
\]

Now consider the long exact sequence
\begin{equation}\label{eq:first_Chern}
\xymatrix{
\cdots\ar[r]&H^1(\overline{X},\mathcal{O}_{\overline{X}})\ar[r]
&H^1(\overline{X},\mathcal{O}^*_{\overline{X}})\ar[r]^-{c_1}&H^2(\overline{X},\Z )\ar[r]&\cdots .
}
\end{equation}
It follows from Lemmas \ref{lem:KL1} and \ref{lem:KL2} that 
$H^1(\overline{X},\mathcal{O}_{\overline{X}})\cong H^{0,1}(\overline{X})=0$.
Thus the map $c_1$ in \eqref{eq:first_Chern} 
is injective and so $c_1(K_{\overline{X}}\otimes [D])=0$ implies 
$K_{\overline{X}}\otimes [D]\cong\mathcal{O}_{\overline{X}}$.
Hence $D$ is an anticanonical divisor on $\overline{X}$.
\end{proof}

Therefore the above constructed pair $(\overline X,D)$ is an
admissible pair, which is said to be of \textit{non-symplectic type}
except the case of  $(r,a,\delta)=(10,10,0)$. 
In order to show the main result Proposition \ref{prop:betti non-symp} in this subsection, we require the
following. 
\begin{proposition}[Kovalev-Lee \cite{KL11}, Proposition $4.3$]\label{prop:KL}\qquad
\begin{enumerate}
\item[(i)] $h^{1,1}(\overline{X})=b^2(\overline{X})=3+2r-a$ \quad and \quad $h^{1,2}(\overline{X})=\frac{1}{2}b^3(\overline{X})=22-r-a$.
\item[(ii)] For the restriction map $\iota' : H^2 (\overline{X}, \R) \longrightarrow H^2(D,\R)$ given by
\begin{equation}\label{map:restrict}
\iota' : H^2 (\overline{X}, \R) \longrightarrow H^2(D,\R), \qquad [\omega] \longmapsto [\restrict{\omega}{D}],
\end{equation}
we have $\dim_{\R} \Image \iota' =r$.
\end{enumerate}
\end{proposition}

\begin{proposition}\label{prop:betti non-symp}
Let $(S, \rho)$ be a $K3$ surface with a non-symplectic involution $\rho$ which is determined by a $K3$ invariant
$(r,a,\delta)$ up to a deformation. Let $(\overline{X},D)$ be the
admissible pair of non-symplectic type obtained in the above
construction from $(S,\rho)$. Let $M$ denote the Calabi-Yau
threefold constructed from two copies of $(\overline{X},D)$ by
Corollary $\ref{cor:doubling}$. Then the number of possibilities of the $K3$ invariants is $75$.
The number of topological types of $(\overline{X}, D)$ which are distinguished by Betti or Hodge numbers is $64$.
Moreover, we have
\[
\begin{cases}
h^{1,1}(M)=b^2(M)=5+3r-2a, \\
h^{2,1}(M)=\frac{1}{2}b^3(M)-1=65-3r-2a . \\
\end{cases}
\]
\end{proposition}

\begin{proof}
Recall that we set $d=\dim_{\R} \Ker \iota$, where 
\[
\iota : H^2 (X, \R) \longrightarrow H^2(D,\R)
\]
is a homomorphism in (\ref{map:iota}). As in $(4.3)$ in \cite{KL11}, we have
\[
d=\dim_{\R}\Ker \iota = \dim_{\R}\Ker \iota' -1,
\]
where $\iota' : H^2 (\overline{X}, \R) \longrightarrow H^2(D,\R)$ is the restriction map defined in (\ref{map:restrict}).
Since $\dim_{\R} \Image \iota' =r$ by Proposition \ref{prop:KL} (ii), we conclude that
\begin{equation*}
d=b^2(\overline{X})-\dim_{\R}\Image \iota' -1 = h^{1,1}(\overline{X})-r-1.
\end{equation*}
Here we used the equality $h^{2,0}(\overline{X})=0$ given by Proposition $2.2$ in \cite{KL11}. Substituting this into (\ref{eq:betti}) in Proposition \ref{prop:betti}, we have
\begin{align}\label{eq:betti2}
\begin{split}
\begin{cases}
b^2(M)&=  2h^{1,1}(\overline X)-r-1,\\
b^3(M)&= 2(2h^{2,1}(\overline X)+22-r).
\end{cases}
\end{split}
\end{align}
In the above equation, we again used $h^{3,0}(\overline{X})=0$ by Proposition $2.2$ in \cite{KL11}. Now the result follows immediately from Proposition 
\ref{prop:KL} (i). Remark that our result is independent of the integer $\delta$.
\end{proof}

\begin{remark}\rm
We can also compute the Hodge numbers of the resulting Calabi-Yau threefolds using the Chen-Ruan orbifold cohomology.
See \cite{PR12} for more details. However, Prof.~Reidegeld pointed out in a private communication that there is another technical problem in the case of non-symplectic automorphisms of order $3\leqslant p \leqslant 19$. More precisely, the $K3$ divisors of the compact K\"ahler threefolds which they have constructed
in \cite{PR12} are in the $p/2$-multiple of the anticanonical class. This implies that a Ricci-flat K\"ahler form on $X=\overline{X}\setminus D$ is not asymptotically cylindrical but asymptotically conical. Therefore, their examples of admissible pairs are not applicable to our doubling construction. However, this problem does not affect the method of calculating the Hodge numbers of the resulting Calabi-Yau threefolds, and so an analogous argument of Proposition \ref{prop:betti non-symp}  will work.
\end{remark}

\newpage

\section{Appendix: The list of the resulting Calabi-Yau threefolds}
In this section, we list all Calabi-Yau threefolds obtained in
Corollary $\ref{cor:doubling}$. We have the following two choices for
constructing Calabi-Yau threefolds $M$:
\begin{enumerate}
\item[(a)] We shall use admissible pairs of \textit{Fano type}. From a Fano threefold $V$, we obtain
an admissible pair $(\overline{X},D)$ by Theorem \ref{th:ad Fano}.
According to the complete classification of nonsingular Fano threefolds \cite{I78,
MM81,MM03}, there are $105$ algebraic families with Picard number
$1\leqslant \rho(V)\leqslant 10$. Then the number of
distinct topological types of the resulting Calabi-Yau threefolds is $59$ (see
Table $6.1$, and also Figure $6.3$ where the resulting Calabi-Yau threefolds are plotted with symbol \scalebox{0.8}{$\times$}).

\item[(b)] We shall use admissible pairs of \textit{non-symplectic type}. 
Starting from a $K3$ surface $S$ with a non-symplectic involution $\rho$,
we obtain an admissible pair  $(\overline X, D)$ as in Section \ref{sec:non-symplectic}.
According to the classification result of $(S,\rho)$ due to Nikulin \cite{N79,N80,N83},
there are $74$ algebraic families. 
Then the number of distinct topological types of the resulting Calabi-Yau threefolds
is $64$. Of these Calabi-Yau threefolds, there is at least one new example
which is not diffeomorphic to the known ones (see Table $6.2$, and also Figure $6.3$ where the
resulting Calabi-Yau threefolds are plotted with symbols $\bullet$ and \,\scalebox{0.6}{$\blacksquare$}\,).
\end{enumerate}

\subsection{All possible Calabi-Yau threefolds from Fano type}
In Table $6.1$, we hereby list the details of the resulting Calabi-Yau
threefolds $M$ from admissible pairs of Fano type. These topological
invariants are computable by Proposition $\ref{prop:betti Fano
type}$, and further details are left to the reader. In the table
below, $\rho=\rho(V)$ denotes the Picard number of the Fano
threefold $V$, and $h^{1,1}=h^{1,1}(M),\; h^{2,1}=h^{2,1}(M)$ denote
the Hodge numbers.
$$ $$
\noindent\begin{center}
 \begin{tabular}{c|c|c|c|c}
 \multicolumn{5}{c}{\vspace{0.3cm}{\textit{Fano threefolds with $\rho=1$} }}                                   \\
 \vspace{-0.23cm}                                &   Label                                      &                    &                              &      \\
         \vspace{-0.23cm}   No.                 &                                         &  $-K_V^3$   &     $h^{1,2}(V)$          &  $(h^{1,1},h^{2,1})$    \\
                                                             &    in \cite{MM81}                      &                     &                            &     \\
 \hline
 \phantom{$1$}$1$                                                     &  $-$                                          &  $2$             & $52$              & $(2,128)$    \\
 \phantom{$1$}$2$                                                     &   $-$                                         &   $4$            & $30$             & $(2,86)$    \\
 \phantom{$1$}$3$                                                     &   $-$                                         &   $6$             & $20$             & $(2,68)$    \\
 \phantom{$1$}$4$                                                     &  $-$                                           &  $8$             & $14$             & $(2,58)$    \\
 \phantom{$1$}$5$                                                     &  $-$                                          &  $10$            & $10$              & $(2,52)$    \\
 \phantom{$1$}$6$                                                     &  $-$                                          &  $12$             & $7$              & $(2,48)$    \\
 \phantom{$1$}$7$                                                     &   $-$                                         &   $14$            & $5$             & $(2,46)$    \\
 \phantom{$1$}$8$                                                     &   $-$                                         &   $16$             & $3$             & $(2,44)$    \\
 \phantom{$1$}$9$                                                     &  $-$                                           &  $18$             & $2$             & $(2,44)$    \\
 $10$                                                     &  $-$                                          &  $22$            & $0$              & $(2,44)$    \\
 $11$                                                     &  $-$                                          &  $8$             & $21$              & $(2,72)$    \\
 $12$                                                     &   $-$                                         &   $16$            & $10$             & $(2,58)$    \\
 $13$                                                     &   $-$                                         &   $24$             & $5$             & $(2,56)$    \\
 $14$                                                     &  $-$                                           &  $32$             & $2$             & $(2,58)$    \\
 $15$                                                     &  $-$                                          &  $40$            & $0$              & $(2,62)$    \\
 $16$                                                     &  $-$                                          &  $54$             & $0$              & $(2,76)$    \\
 \vspace{2.4cm}
 \hspace{-0.1cm}$17$                                                     &   $-$                                         &   $64$            & $0$             & $(2,86)$    \\
  \end{tabular}\hfill
 \begin{tabular}{c|c|c|c|c}
 \multicolumn{5}{c}{\vspace{0.3cm}{\textit{Fano threefolds with $\rho=2$}} }                                   \\
 \vspace{-0.23cm}                                &   Label                                      &                    &                              &      \\
         \vspace{-0.23cm}   No.                 &                                         &  $-K_V^3$   &     $h^{1,2}(V)$          &  $(h^{1,1},h^{2,1})$    \\
                                                             &    in \cite{MM81}                      &                     &                            &     \\
 \hline
 $18$                                                     &  \phantom{$1$}$1$                                          &  $4$             & $22$              & $(3,69)$    \\
 $19$                                                     &   \phantom{$1$}$2$                                         &   $6$            & $20$             & $(3,67)$    \\
 $20$                                                     &   \phantom{$1$}$3$                                         &   $8$             & $11$             & $(3,51)$    \\
 $21$                                                     &  \phantom{$1$}$4$                                           &  $10$             & $10$            & $(3,51)$    \\
 $22$                                                     &  \phantom{$1$}$5$                                          &  $12$            & $6$              & $(3,45)$    \\
 $23$                                                     &  \phantom{$1$}$6$                                          &  $12$             & $9$              & $(3,51)$    \\
 $24$                                                     &   \phantom{$1$}$7$                                         &   $14$            & $5$             & $(3,45)$    \\
 $25$                                                     &   \phantom{$1$}$8$                                         &   $14$             & $9$             & $(3,53)$    \\
 $26$                                                     &  \phantom{$1$}$9$                                           &  $16$             & $5$             & $(3,47)$    \\
 $27$                                                     &  $10$                                          &  $16$            & $3$              & $(3,43)$    \\
 $28$                                                     &  $11$                                          &  $18$             & $5$              & $(3,49)$    \\
 $29$                                                     &   $12$                                         &   $20$            & $3$             & $(3,47)$    \\
 $30$                                                     &   $13$                                         &   $20$             & $2$             & $(3,45)$    \\
 $31$                                                     &  $14$                                           &  $20$             & $1$             & $(3,43)$    \\
 $32$                                                     &  $15$                                          &  $22$            & $4$              & $(3,51)$    \\
 $33$                                                     &  $16$                                          &  $22$             & $2$              & $(3,47)$    \\
 $34$                                                     &   $17$                                         &   $24$            & $1$             & $(3,47)$    \\
 $35$                                                     &   $18$                                         &   $24$            & $2$             & $(3,49)$    \\
 $36$                                                     &   $19$                                          &  $26$             & $2$              & $(3,51)$    \\
 $37$                                                     &   $20$                                         &   $26$            & $0$             & $(3,47)$    \\
 $38$                                                     &   $21$                                         &   $28$             & $0$             & $(3,49)$    \\
 $39$                                                     &  $22$                                           &  $30$             & $0$            & $(3,51)$    \\
\end{tabular}\hfill
\end{center}
\noindent\begin{center}
 \begin{tabular}{c|c|c|c|c}
 \vspace{-0.23cm}                                &   Label                                      &                    &                              &      \\
         \vspace{-0.23cm}   No.                 &                                         &  $-K_V^3$   &     $h^{1,2}(V)$          &  $(h^{1,1},h^{2,1})$    \\
                                                             &    in \cite{MM81}                      &                     &                            &     \\
 \hline
 $40$                                                     &  $23$                                          &  $30$            & $1$              & $(3,53)$    \\
 $41$                                                     &  $24$                                          &  $30$             & $0$              & $(3,51)$    \\
 $42$                                                     &   $25$                                         &   $32$            & $1$             & $(3,55)$    \\
 $43$                                                     &   $26$                                         &   $34$             & $0$             & $(3,55)$    \\
 $44$                                                     &  $27$                                           &  $38$             & $0$             & $(3,59)$    \\
 $45$                                                     &  $28$                                          &  $40$            & $1$              & $(3,63)$    \\
 $46$                                                     &  $29$                                          &  $40$             & $0$              & $(3,61)$    \\
 $47$                                                     &   $30$                                         &   $46$            & $0$             & $(3,67)$    \\
 $48$                                                     &   $31$                                         &   $46$             & $0$             & $(3,67)$    \\
 $49$                                                     &  $32$                                           &  $48$             & $0$             & $(3,69)$    \\
 $50$                                                     &  $33$                                          &  $54$            & $0$              & $(3,75)$    \\
 $51$                                                     &  $34$                                          &  $54$             & $0$              & $(3,75)$    \\
 $52$                                                     &   $35$                                         &   $56$            & $0$             & $(3,77)$    \\
 \vspace{0.5cm}
 \hspace{-0.1cm}\noindent$53$                                                     &   $36$                                         &   $62$            & $0$             & $(3,83)$    \\
  \multicolumn{5}{c}{\vspace{0.3cm}{\textit{Fano threefolds with $\rho=3$} }}                                   \\
 \vspace{-0.23cm}                                &   Label                                      &                    &                              &      \\
         \vspace{-0.23cm}   No.                 &                                         &  $-K_V^3$   &     $h^{1,2}(V)$          &  $(h^{1,1},h^{2,1})$    \\
                                                             &    in \cite{MM81}                      &                     &                            &     \\
 \hline
 $54$                                                     &  \phantom{$1$}$1$                                          &  $12$             & $8$              & $(4,48)$    \\
 $55$                                                     &   \phantom{$1$}$2$                                         &   $14$            & $3$             & $(4,40)$    \\
 $56$                                                     &   \phantom{$1$}$3$                                         &   $18$             & $3$             & $(4,44)$    \\
 $57$                                                     &  \phantom{$1$}$4$                                           &  $18$             & $2$             & $(4,42)$    \\
 $58$                                                     &  \phantom{$1$}$5$                                          &  $20$            & $0$              & $(4,40)$    \\
 $59$                                                     &  \phantom{$1$}$6$                                          &  $22$             & $1$              & $(4,44)$    \\
 $60$                                                     &   \phantom{$1$}$7$                                         &   $24$            & $1$             & $(4,46)$    \\
 $61$                                                     &  \phantom{$1$}$8$                                         &   $24$             & $0$             & $(4,44)$    \\
 $62$                                                     &  \phantom{$1$}$9$                                           &  $26$             & $3$             & $(4,52)$    \\
 $63$                                                     &  $10$                                          &  $26$            & $0$              & $(4,46)$    \\
 $64$                                                     &  $11$                                          &  $28$             & $1$              & $(4,50)$    \\
 $65$                                                     &   $12$                                         &   $28$            & $0$             & $(4,48)$    \\
 $66$                                                     &   $13$                                         &   $30$             & $0$             & $(4,50)$    \\
 $67$                                                     &  $14$                                           &  $32$             & $1$             & $(4,54)$    \\
 $68$                                                     &  $15$                                          &  $32$            & $0$              & $(4,52)$    \\
 $69$                                                     &  $16$                                          &  $34$             & $0$              & $(4,54)$    \\
 $70$                                                     &   $17$                                         &   $36$            & $0$             & $(4,56)$    \\
 $71$                                                     &  $18$                                          &  $36$             & $0$              & $(4,56)$    \\
 $72$                                                     &   $19$                                         &   $38$            & $0$             & $(4,58)$    \\
 $73$                                                     &   $20$                                         &   $38$             & $0$             & $(4,58)$    \\
 $74$                                                     &  $21$                                           &  $38$             & $0$             & $(4,58)$    \\
 $75$                                                     &  $22$                                          &  $40$            & $0$              & $(4,60)$    \\
 $76$                                                     &  $23$                                          &  $42$             & $0$              & $(4,62)$    \\
 $77$                                                     &   $24$                                         &   $42$            & $0$             & $(4,62)$    \\
 $78$                                                     &   $25$                                         &   $44$             & $0$             & $(4,64)$    \\
 $79$                                                     &  $26$                                           &  $46$             & $0$             & $(4,66)$    \\
 $80$                                                     &  $27$                                          &  $48$            & $0$              & $(4,68)$    \\
 $81$                                                     &  $28$                                          &  $48$             & $0$              & $(4,68)$    \\
 $82$                                                     &   $29$                                         &   $50$            & $0$             & $(4,70)$    \\
 $83$                                                     &   $30$                                         &   $50$             & $0$             & $(4,70)$    \\
 \vspace{0.5cm}
 \hspace{-0.1cm}$84$                                                     &  $31$                                           &  $52$             & $0$             & $(4,72)$    \\
 \end{tabular}\hfill
 \begin{tabular}{c|c|c|c|c}
 \multicolumn{5}{c}{\vspace{0.3cm}{\textit{Fano threefolds with $\rho=4$}} }                                   \\
 \vspace{-0.23cm}                                &   Label                                      &                    &                              &      \\
         \vspace{-0.23cm}   No.                 &                                         &  $-K_V^3$   &     $h^{1,2}(V)$          &  $(h^{1,1},h^{2,1})$    \\
                                                             &    in \cite{MM81}                      &                     &                            &     \\
 \hline
 $85$                                                     &  \phantom{$1$}$1$                                          &  $24$             & $1$              & $(5,45)$    \\
 $86$                                                     &   \phantom{$1$}$2$                                         &   $28$            & $1$             & $(5,49)$    \\
 $87$                                                     &   \phantom{$1$}$3$                                         &   $30$             & $0$             & $(5,49)$    \\
 $88$                                                     &  \phantom{$1$}$4$                                           &  $32$             & $0$            & $(5,51)$    \\
 $89$                                                     &  \phantom{$1$}$5$                                          &  $32$            & $0$              & $(5,51)$    \\
 $90$                                                     &  \phantom{$1$}$6$                                          &  $34$             & $0$              & $(5,53)$    \\
 $91$                                                     &   \phantom{$1$}$7$                                         &   $36$            & $0$             & $(5,55)$    \\
 $92$                                                     &   \phantom{$1$}$8$                                         &   $38$             & $0$             & $(5,57)$    \\
 $93$                                                     &  \phantom{$1$}$9$                                           &  $40$             & $0$             & $(5,59)$    \\
 $94$                                                     &  $10$                                          &  $42$            & $0$              & $(5,61)$    \\
 $95$                                                     &  $11$                                          &  $44$             & $0$              & $(5,63)$    \\
 $96$                                                     &   $12$                                         &   $46$            & $0$             & $(5,65)$    \\
 \vspace{0.5cm}
 \hspace{0.01cm}
 $97^*$                                     &   $-$                                               &   $26$             & $0$             & $(5,45)$    \\
 \multicolumn{5}{l}{$\ast)\:$ No. $97$ was erroneously omitted in \cite{MM81}.}   \\
 \multicolumn{5}{l}{See \cite{MM03} for the correct table.} \vspace{1.2cm}  \\
 \multicolumn{5}{c}{\vspace{0.3cm}{\textit{Fano threefolds with $\rho\geqslant 5$}} }                                   \\
 \vspace{-0.23cm}                                &                                         &                    &                              &      \\
         \vspace{-0.23cm}   No.                 &    $\rho$                                     &  $-K_V^3$   &     $h^{1,2}(V)$          &  $(h^{1,1},h^{2,1})$    \\
                                                             &                           &                     &                            &     \\
 \hline
 \phantom{$1$}$98$                                         &  $5$                                           &  $28$             & $0$             & $(6,46)$    \\
 \phantom{$1$}$99$                                                     &  $5$                                          &  $36$            & $0$              & $(6,54)$    \\
 \hspace{0.01cm}
 $100^{\dagger}$                                     &  $5$                                          &  $36$             & $0$              & $(6,54)$    \\
 $101$                                                     &   $6$                                         &   $30$            & $0$             & $(7,47)$    \\
 $102$                                                     &   $7$                                         &   $24$            & $0$             & $(8,40)$    \\
 $103$                                                     &   $8$                                          &  $18$             & $0$              & $(9,33)$    \\
 $104$                                                     &   $9$                                         &   $12$            & $0$             & $(10,26)$    \\
 \vspace{0.5cm}
 \hspace{-0.1cm}$105$                                                     &   $10$                                         &   $6$             & $0$             & $(11,19)$    \\
 \multicolumn{5}{l}{$\dagger)\:$ This Fano threefold is $\C P^1\times S_6$ where}   \\
  \multicolumn{5}{l}{$S_6$ is a del Pezzo surface of degree $6$.}
 \end{tabular}\hfill 
 Table $6.1$. The list of Calabi-Yau threefolds from Fano type
\end{center}

\subsection{All possible Calabi-Yau threefolds from non-symplectic type}\label{subset:non-symp}
In Table $6.2$, we hereby list the details of the resulting Calabi-Yau
threefolds from admissible pairs of non-symplectic type. These Hodge
numbers are also computable by Proposition $\ref{prop:betti
non-symp}$ and further details are left to the reader. In the table
below, there is at least {\it one} new example of Calabi-Yau
threefolds, which is listed as the boxed number $64$.  
We also list the number of the mirror partner for each resulting Calabi-Yau threefold in our construction.
See Discussion and  Section \ref{subsec: graphix} below for more details.
The symbol -- on the list means that the corresponding Calabi-Yau
threefold has no mirror partner in this construction.
\vspace{0.3cm}
\noindent\begin{center}
 \begin{tabular}{c|c|c|c}
 \multicolumn{4}{c}{\vspace{0.3cm}{\textit{$K3$ surfaces with non-symplectic involutions} }}                                   \\
 \vspace{-0.23cm}                                &   $K3$ invariants                          &                    &       Mirror          \\
         \vspace{-0.23cm}   No.                 &                                          &  $(h^{1,1},h^{2,1})$ &                     \\
                                                             &    $(r,a,\delta)$                        &                           &     partner      \\
 \hline
\phantom{$1$}$1$                                                     &  $(2,0,0)$        &  $(11, 59)$     &  \phantom{$1$}$3$   \\
 \phantom{$1$}$2$                                                     &  $(10,0,0)$        &  $(35,35)$     & \phantom{$1$}$2$    \\
 \phantom{$1$}$3$                                                     &  $(18,0,0)$        &  $(59,11)$      & \phantom{$1$}$1$  \\
 \phantom{$1$}$4$                                                     &  $(1,1,1)$        &  \;$(6,60)$          & \phantom{$1$}$9$  \\
 \phantom{$1$}$5$                                                     &  $(3,1,1)$        &  $(12,54)$        & \phantom{$1$}$8$  \\
 \phantom{$1$}$6$                                                     &  $(9,1,1)$        &  $(30,36)$        & \phantom{$1$}$7$ \\
 \phantom{$1$}$7$                                                     &  $(11,1,1)$        &  $(36,30)$         & \phantom{$1$}$6$  \\
 \phantom{$1$}$8$                                                     &  $(17,1,1)$        &  $(54,12)$         & \phantom{$1$}$5$  \\
 \phantom{$1$}$9$                                                     &  $(19,1,1)$        &  $(60,6)$        & \phantom{$1$}$4$  \\
 $10$                                                     &  $(2,2,0 \:\text{or}\: 1)$        &  $(7,55)$     & $18$  \\
 $11$                                                     &  $(4,2,1)$        &  $(13,49)$           & $17$  \\
 $12$                                                     &  $(6,2,0)$        &  $(19,43)$          & $16$  \\
 $13$                                                     &  $(8,2,0)$        &  $(25,37)$         & $15$  \\
 $14$                                                     &  $(10,2,0 \:\text{or}\: 1)$        &  $(31,31)$  &  $14$  \\
 $15$                                                     &  $(12,2,1)$        &  $(37,25)$  &  $13$  \\
 $16$                                                     &  $(14,2,0)$        &  $(43,19)$  &  $12$  \\
 $17$                                                     &  $(16,2,1)$        &  $(49,13)$  &  $11$  \\
$18$                                                     &  $(18,2,0 \:\text{or}\: 1)$        &  $(55,7)$   & $10$  \\
 $19$                                                     &  $(20,2,1)$        &  $(61,1)$     & --  \\
 $20$                                                     &  $(3,3,1)$        &  $(8,50)$      & $27$  \\
 $21$                                                     &  $(5,3,1)$        &  $(14,44)$      & $26$  \\
 $22$                                                     &  $(7,3,1)$        &  $(20,38)$       & $25$  \\
 $23$                                                     &  $(9,3,1)$        &  $(26,32)$       & $24$  \\
 $24$                                                     &  $(11,3,1)$        &  $(32,26)$      & $23$  \\
 $25$                                                     &  $(13,3,1)$        &  $(38,20)$      & $22$  \\
 $26$                                                     &  $(15,3,1)$        &  $(44,14)$       & $21$  \\
 $27$                                                     &  $(17,3,1)$        &  $(50,8)$         & $20$  \\
 $28$                                                     &  $(19,3,1)$        &  $(56,2)$        & --  \\
 $29$                                                     &  $(4,4,1)$        &  $(9,45)$         & $35$  \\
 $30$                                                     &  $(6,4,0 \:\text{or}\: 1)$        &  $(15,39)$      & $34$  \\
 $31$                                                     &  $(8,4,1)$        &  $(21,33)$         & $33$  \\
 $32$                                                     &  $(10,4,0 \:\text{or}\: 1)$        &  $(27,27)$       & $32$  \\
 $33$                                                     &  $(12,4,1)$        &  $(33,21)$         & $31$  \\
 $34$                                                     &  $(14,4,0 \:\text{or}\: 1)$        &  $(39,15)$       & $30$  \\
 $35$                                                     &  $(16,4,1)$        &  $(45,9)$      & $29$  \\
   \end{tabular}\hfill
\begin{tabular}{c|c|c|c}
 \vspace{-0.23cm}                                &   $K3$ invariants                          &                    &       Mirror          \\
         \vspace{-0.23cm}   No.                 &                                          &  $(h^{1,1},h^{2,1})$ &                     \\
                                                             &    $(r,a,\delta)$                        &                           &     partner      \\
 \hline
 $36$                                                     &  $(18,4,0 \:\text{or}\: 1)$        &  $(51,3)$        & --  \\
 $37$                                                     &  $(5,5,1)$        &  $(10,40)$        & $42$  \\
 $38$                                                     &  $(7,5,1)$        &  $(16,34)$        & $41$  \\
 $39$                                                     &  $(9,5,1)$        &  $(22,28)$        & $40$  \\
$40$                                                     &  $(11,5,1)$        &  $(28,22)$        & $39$  \\
$41$                                                     &  $(13,5,1)$        &  $(34,16)$     &  $38$  \\
 $42$                                                     &  $(15,5,1)$        &  $(40,10)$     & $37$  \\
 $43$                                                     &  $(17,5,1)$        &  $(46,4)$       & --  \\
 $44$                                                     &  $(6,6,1)$        &  $(11,35)$        & $48$  \\
 $45$                                                     &  $(8,6,1)$        &  $(17,29)$        & $47$  \\
 $46$                                                     &  $(10,6,0\:\text{or}\: 1)$        &  $(23,23)$    & $46$  \\
 $47$                                                     &  $(12,6,1)$        &  $(29,17)$   & $45$  \\
 $48$                                                     &  $(14,6,0\:\text{or}\: 1)$        &  $(35,11)$       & $44$  \\
 $49$                                                     &  $(16,6,1)$        &  $(41,5)$         & --  \\
 $50$                                                     &  $(7,7,1)$        &  $(12,30)$        & $53$  \\
 $51$                                                     &  $(9,7,1)$        &  $(18,24)$        & $52$  \\
 $52$                                                     &  $(11,7,1)$        &  $(24,18)$       & $51$  \\
 $53$                                                     &  $(13,7,1)$        &  $(30,12)$       & $50$  \\
 $54$                                                     &  $(15,7,1)$        &  $(36,6)$         & --  \\
 $55$                                                     &  $(8,8,1)$        &  $(13,25)$         & $57$  \\
 $56$                                                     &  $(10,8,0\:\text{or}\: 1)$        &  $(19,19)$         & $56$  \\
 $57$                                                     &  $(12,8,1)$        &  $(25,13)$       & $55$  \\
 $58$                                                     &  $(14,8,1)$        &  $(31,7)$        &  --  \\
 $59$                                                     &  $(9,9,1)$        &  $(14,20)$       & $60$  \\
 $60$                                                     &  $(11,9,1)$        &  $(20,14)$      & $59$  \\
 $61$                                                     &  $(13,9,1)$        &  $(26,8)$        & --  \\
 $62$                                                     &  $(10,10,1)^{\natural}$     &  $(15,15)$       & $62$  \\
 $63$                                                     &  $(12,10,1)$        &  $(21,9)$       & --  \\
 \vspace{0.5cm}                                                   
\hspace{-0.1cm}\fbox{$64$}&  $(11,11,1)$        &  $(16,10)$       & --  \\
 \multicolumn{4}{l}{$\natural)$  $\;(r,a,\delta)\neq(10,10,0)$ from assumption.} \vspace{1.00cm}  \\
 \end{tabular}
 \end{center}
\begin{center}\vspace{0.4cm}
Table $6.2$. The list of Calabi-Yau threefolds from non-symplectic type
\end{center}


\pagebreak

\noindent\textit{Discussion.}
 \noindent The method of constructing Calabi-Yau threefolds and their mirrors from $K3$ surfaces were originally 
 investigated by Borcea and Voisin 
 \cite{BV97}, Section $4$, using algebraic geometry. Our doubling construction is a differential-geometric 
 interpretation of the Borcea-Voisin construction.
 Observe that Proposition \ref{prop:betti non-symp} gives the condition that two Calabi-Yau
threefolds $M$ and $M'$ should be a mirror pair, i.e.,
$h^{p,q}(M)=h^{3-p,q}(M')$ for all $p,q\in\set{0,1,2,3}$.  
Let $M$ (resp. $M'$) be a Calabi-Yau threefold
from admissible pairs of non-symplectic type with respect to $K3$
invariants $(r,a,\delta)$ (resp. $(r',a', \delta')$). Then
$h^{p,q}(M)=h^{3-p,q}(M')$ implies $r+r'=20,\; a=a'$ by Proposition
\ref{prop:betti non-symp}. These relations coincide with $(11)$ in \cite{BV97}, p.723. From these equalities, we can find mirror
pairs in our examples of Calabi-Yau threefolds. In particular $M$
is automatically self-mirror when $r=10$.
Thus we find 24 mirror pairs and 6 self-mirror Calabi-Yau threefolds
in our examples.

\subsection{Graphical chart of our examples}\label{subsec: graphix}
Finally we plot the Hodge numbers of the resulting Calabi-Yau
threefolds  in Figure $6.3$. In the following figure,
the Calabi-Yau threefolds obtained from Fano type (case (a)) are
registered as symbol \scalebox{0.8}{$\times$} and those from non-symplectic type (case (b))
are registered as symbol $\bullet$. Separately, our new
example is denoted by solid square \scalebox{0.6}{$\blacksquare$} in Figure $6.3$. We
take the Euler characteristic $\chi=2(h^{1,1}-h^{2,1})$ along the
$X$-axis and $h^{1,1}+h^{2,1}$ along the $Y$-axis.
We see that all our examples from non-symplectic type
are located on the integral lattice of the form
\begin{equation}\label{eq:nonsymp_grid}
(X,Y)=(12,26)+m(12,4)+n(-12,4),\quad m,n\in\Z_{\geqslant 0}.
\end{equation}
In this plot the mirror symmetry is considered as the inversion 
$\mu :(X,Y)\longmapsto (-X,Y)$
with respect to the $Y$-axis.
The set of $54$ points with $n>0$ in \eqref{eq:nonsymp_grid} is $\mu$-invariant,
and thus the corresponding Calabi-Yau threefolds have a mirror partner in this set.

{\vspace{-0.1cm}{
\begin{figure}[h]\label{fig:CY}
\centering
\includegraphics[width=\hsize ,clip]{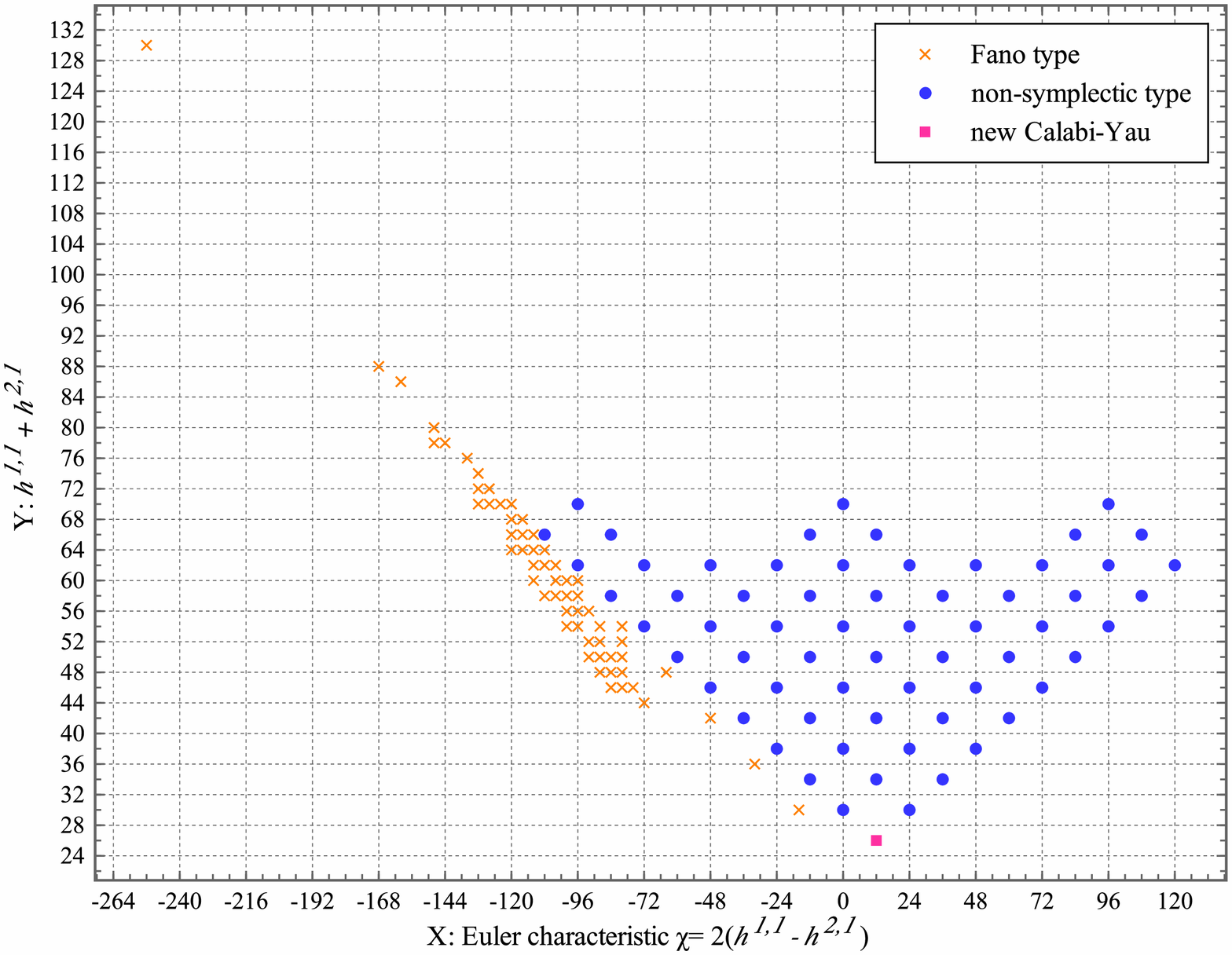}\vspace{0.2cm}
Figure $6.3$. All resulting Calabi-Yau threefolds
\end{figure}
}}

\end{document}